\documentclass[10pt]{article}
\usepackage[a4paper,
  left=3.2cm,
  right=3.2cm,
  top=3.0cm,
  bottom=3.0cm]{geometry}
\usepackage{xcolor}
\usepackage[utf8]{inputenc}
\usepackage[T1]{fontenc}
\definecolor{mg}{rgb}  {0.85, 0.,  0.85}

\newcommand\scalemath[2]{\scalebox{#1}{\mbox{\ensuremath{\displaystyle #2}}}}

\usepackage[UKenglish]{babel}
\usepackage{amsmath,amssymb,amsthm}
\usepackage{mathrsfs}
\usepackage{verbatim}
\usepackage{newtxtext,newtxmath}
\usepackage{bm}
\usepackage{microtype}
\microtypesetup{protrusion=true,expansion=true}
\usepackage{graphicx}
\usepackage{caption}
\usepackage{subcaption}
\usepackage{bm}
\usepackage{dsfont}
\usepackage[normalem]{ulem} 
\usepackage[retainorgcmds]{IEEEtrantools}
\usepackage{tikz}
\usepackage{tikz-cd}
\usetikzlibrary{arrows}
\usetikzlibrary{arrows.meta,decorations.markings,patterns}
\usetikzlibrary{shapes.geometric}
\usetikzlibrary{backgrounds}
\tikzset{
  midarrow/.style={
    postaction={decorate},
    decoration={
      markings,
      mark=at position 0.5 with {\arrow{Stealth}}}}}
\usetikzlibrary{calc} 
\usetikzlibrary{positioning}

\usepackage{csquotes}
\usepackage{mathtools}
\usepackage{xfrac}
\usepackage[colorlinks]{hyperref}
\usepackage[capitalise,nameinlink]{cleveref}
\crefformat{equation}{(#2#1#3)}
\crefrangeformat{equation}{(#3#1#4) to~(#5#2#6)}
\crefmultiformat{equation}{(#2#1#3)}%
{ and~(#2#1#3)}{, (#2#1#3)}{ and~(#2#1#3)}
\makeatletter
\newcommand{\refcheckize}[1]{%
  \expandafter\let\csname @@\string#1\endcsname#1%
  \expandafter\DeclareRobustCommand\csname relax\string#1\endcsname[1]{%
    \csname @@\string#1\endcsname{##1}\@for\@temp:=##1\do{\wrtusdrf{\@temp}\wrtusdrf{{\@temp}}}}%
  \expandafter\let\expandafter#1\csname relax\string#1\endcsname
}
\newcommand{\refcheckizetwo}[1]{%
  \expandafter\let\csname @@\string#1\endcsname#1%
  \expandafter\DeclareRobustCommand\csname relax\string#1\endcsname[2]{%
    \csname @@\string#1\endcsname{##1}{##2}\wrtusdrf{##1}\wrtusdrf{{##1}}\wrtusdrf{##2}\wrtusdrf{{##2}}}%
  \expandafter\let\expandafter#1\csname relax\string#1\endcsname
}
\makeatother

\newcommand\ie{{\em i.e.}~}

\newcommand\N{\mathbb{N}}
\newcommand\T{\mathbb{T}}

\newcommand\R{\mathbb{R}}
\newcommand\Z{\mathbb{Z}}
\newcommand\C{\mathbb{C}}
\newcommand\U{\mathscr U}
\newcommand\V{\mathcal V}
\newcommand\A{\mathcal A} 
\newcommand\F{\mathcal F}
\newcommand\E{\mathcal E} 
\newcommand\CC{\mathcal C}
\newcommand\x{\mathrm x}
\newcommand\y{\mathrm y}
\newcommand\e{\mathrm e}
\newcommand\f{\mathrm f}
\newcommand\dd{\mathrm d}
\newcommand\xx{\mathfrak x}
\newcommand\ee{\mathfrak e}
\newcommand\ff{\mathfrak f}
\newcommand\XX{\mathfrak X}

\def\H{\mathcal H}
\def\nv{\mathscr{v}}
\def\ne{\mathscr{e}}
\def\nf{\mathscr{f}}
\def\hx{\hat{\mathfrak{x}}}
\def\he{\hat{\mathfrak{e}}}
\def\hf{\hat{\mathfrak{f}}}

\DeclareMathOperator{\rank}{rank}
\DeclareMathOperator{\Ran}{Ran}
\DeclareMathOperator{\Ker}{Ker}

\newtheorem{Theorem}{Theorem}[section]
\newtheorem{Remark}[Theorem]{Remark}
\newtheorem{Lemma}[Theorem]{Lemma}

\newtheorem{Proposition}[Theorem]{Proposition}
\newtheorem{Definition}[Theorem]{Definition}
\newtheorem{Example}[Theorem]{Example}
\crefname{Lemma}{Lemma}{Lemmas}
\Crefname{Lemma}{Lemma}{Lemmas}
\begin{document}

\title{
Dirac operators on periodic planar graphs: \\
the right definition advocated by flat bands
}


\author{K. Kato${}^1$, D. Parra${}^2$\footnote{Corresponding author: dparra@ubiobio.cl }${\ \;\! }$\footnote{Supported by ANID, Chile, through Subvención a la Instalación en la Academia, Convocatoria 2025, grant N° 82520044, and through FONDECYT de Iniciación grant N° 11261857}, S. Richard${}^3$\footnote{Supported by JSPS Grant-in-Aid for scientific research C no 26K06826}}

\date{\today}

\maketitle 

\begin{quote}
\begin{itemize}
\item[1] The Open University of Japan, 2-11 Wakaba, Mihama-ku, Chiba 261-8586, Japan
\item[2] Departamento de Matemática, Facultad de Ciencias, Universidad del Bío-Bío, Concepción, Chile
\item[3] Institute for Liberal Arts and Sciences \& Graduate School of Mathematics, Nagoya University, Furo-cho, Chikusa-ku, Nagoya 464-8601, Japan
\item[] E-mail: 2610089962@campus.ouj.ac.jp,  \\ 
${}\qquad \quad\ {}$ dparra@ubiobio.cl, \\ 
${}\qquad \quad\ {}$ richard@math.nagoya-u.ac.jp
\end{itemize}
\end{quote}


\begin{abstract}
Massive discrete Dirac-type operators on periodic planar graphs are investigated. 
Using Floquet--Bloch decomposition, we first show that neglecting the faces 
leads to spurious flat bands that are unrelated to the graph's geometry. 
Then, for the correct Dirac--Hodge operator $D_m$, we establish a spectral reduction expressing 
$\sigma(D_m)$ through the graph Laplacian on vertices $-\Delta_\V$ and the Laplacian on faces $-\Delta_\F$.
As a consequence, $D_m$ carries a flat band precisely when one of these scalar Laplacians does. 
We finally apply this reduction to the study of flat bands of the operator $D_m$ on Archimedean lattices 
and contrast the outcome with that on the Lieb lattice.
\end{abstract}
\bigskip


\section{Introduction}

Discrete differential structures have long linked geometry, topology, and spectral analysis. 
On a graph, the discrete exterior derivative maps functions on vertices to functions on edges.
When faces are considered, it also maps functions on edges to functions on faces, 
a chain that in principle continues to higher-dimensional cells. 
These maps, together with their adjoints, combine into a single first-order operator of Dirac type.
This construction has classical roots in the combinatorial Hodge theory of Eckmann \cite{MR13318} 
and the finite-difference approximation of Hodge theory by Dodziuk \cite{MR407872}.
The same structure underpins the discrete exterior calculus of computational modelling \cite{DKT08}.
More recently, this higher-order viewpoint has been investigated in network science, 
where Bianconi introduced a topological Dirac operator on the signals of nodes, links, and cells \cite{Bianconi_2021}, see also
\cite{Bi24}. The framework extends to arbitrary two-dimensional cell complexes \cite{Bianconi_2024}, 
and hence to the polygonal faces of the tilings considered below. 
We refer to \cite{MR4760798} for a review paper on these constructions.

Our interest in Dirac-type operators is based on the spectral theory of quantum Hamiltonians. 
Recall that the spectral theory of the graph Laplacian, a second-order operator acting on vertices alone, 
is by now a mature subject. On infinite graphs it has been developed systematically
for example in \cite{MR4383783, MR986363}, and there exist numerous works on this topic.
In the periodic setting, a detailed understanding of band structures is available through 
Floquet--Bloch decomposition, see \cite{HN, KS1, KS2, parrarichard, sabri} and references therein. 
On the other hand, spectral investigations on operators acting on higher-degree cochains --- edges and, above all, faces --- 
have appeared only recently in the literature.  
They mainly treat the Hodge Laplacian rather than any first-order Dirac operator.
For example, essential self-adjointness and spectral theory have been established 
for Hodge Laplacians on weighted graphs and on general simplicial complexes \cite{BaKe25, EnJa25x}.

In this work, we deal with the spectral theory of first-order operators acting on planar graphs.
Our approach is based on the discrete differential calculus introduced by Miranda and Parra in \cite{MP23}.
Two massive first-order operators can naturally be built from this differential calculus,
namely the Gauss-Bonnet operator and the Dirac--Hodge operator.
The Gauss--Bonnet operator $G_m$ couples vertices and edges through $\dd_0$ and its adjoint, 
with a mass term taking opposite signs on the two degrees.
In the massless case and for infinite graphs of bounded geometry,
the spectral and scattering theory for this operator was developed in \cite{AT15, Pa17}. 
On the other hand, the Dirac--Hodge operator $D_m$ acts on the full chain of vertices, edges, and faces.
Its expression involves the operators $\dd_0$, $\dd_1$, and their adjoints.
Up to a shift by $m^2$, its square is equal to the direct sum of the Hodge Laplacians on vertices, edges, and faces. 
These two operators are discrete counterparts of the same continuous Dirac operator, 
and differ only in whether the two-dimensional information carried by faces is retained.
In this sense, the passage from a Laplacian to a first-order operator is not unique.

Let us stress that this cochain construction is what makes the operator available on an arbitrary periodic graph. 
Discrete analogues of the Dirac operator that do not rely on the cochain formalism must instead 
approximate the partial derivatives of the continuous operator by difference operators, a choice that is not canonical.
As a consequence, the spectral theory of discrete Dirac operators has largely been confined to $\Z^n$ \cite{MR3242707, MR4589246}.
On the other hand, defining the operator via the coboundary maps, as discrete exterior derivatives, 
removes this ambiguity and makes sense on any graph.
Note that Dirac operators that interlace all degrees have been studied mainly through their continuum limits \cite{Na24, MP23, Exner_2025}.
The relation with their continuum counterparts has drawn recent attention through distinct limiting procedures: 
mesh refinement on $h\Z^n$ \cite{MP23,Na24}, and collapse of a graph-like space onto its skeleton \cite{Exner_2025}. 

Having defined first-order operators on planar graphs, 
let us turn to the second specific topic of this paper: flat bands. 
For periodic lattices, the Floquet--Bloch decomposition provides direct access to the band structure of the operators acting on them.
Flat bands correspond to Bloch bands on which the dispersion relation is constant, 
and are of special interest for condensed-matter physics. Among numerous physical references, 
let us mention \cite{Li89, Julku, experimental}, which are relevant for our investigations.  
On the mathematical side, flat bands of the graph Laplacian on periodic graphs have been studied 
for example in \cite{FL, sabri}. 
For the Archimedean tilings, Peyerimhoff and T\"aufer determine which of these lattices carry a flat band for $-\Delta_\V$ in \cite{PeTa21}, 
and Kerner, T\"aufer and Wintermayr analyse the robustness of these flat bands under variation of periodic weights in \cite{KeTaWi24}.

Since all references mentioned above are concerned with graph Laplacians, 
it is natural to wonder what changes when the same tilings are examined through a first-order operator, 
especially when edges and faces are also involved? 
Does the higher-order information carried by edges and faces produce flat bands absent at the level of vertices? 
For example, what about flat bands for the Dirac--Hodge operator $D_m$ on Archimedean tilings, 
how do they compare with the flat bands for $-\Delta_\V$?

Motivated by these questions, we provide three main results.
The first one corresponds to a spectral reduction: By exploiting the block structure of the Dirac--Hodge operator $D_m$
we show that its spectrum is determined entirely by the graph Laplacian on vertices and the Laplacian on faces.
More precisely, every band of $D_m$ is of the form $\pm\sqrt{\,\lambda+m^2\,}$ for $\lambda$ in the spectrum of 
$-\Delta_\V$ or of $-\Delta_\F$.
In particular, $D_m$ has a flat band precisely when $-\Delta_\V$ or $-\Delta_\F$ does. 
As a consequence and quite surprisingly, a question about a first-order operator on three cochain degrees
is reduced to two independent questions about scalar Laplacians.

Our second main result is a criterion for the choice of the most meaningful operator. 
The previous analysis applied to the Gauss--Bonnet operator $G_m$ reveals a flat band at $-m$ 
of multiplicity equal to the number of faces. 
Thus, the operator $G_m$ exhibits a flat band for all planar periodic graphs (having at least one face), 
regardless of the geometry of the graph. 
Since the Dirac--Hodge operator $D_m$ is free of this spurious flat band,  
we understand that its appearance is an artefact of the truncation of the cochain complex at edges. 
For the purpose of detecting genuine, geometry-dependent flat bands, $D_m$ is the appropriate first-order operator, 
while $G_m$ is not.

The third main result is the thorough investigation of flat bands for the Dirac--Hodge operator acting on Archimedean tilings.
This result is obtained by computing $-\Delta_\F$ in each case and by combining this information with the corresponding result 
on $-\Delta_\V$ determined in \cite{PeTa21}. 
Recall that for Archimedean lattices, only the Kagome and the super-Kagome lattices carry a flat band 
for the vertex Laplacian $-\Delta_\V$. 
It turns out that for $-\Delta_\F$ the same pattern also holds: only the Kagome and the super-Kagome lattices possess flat bands.
Thus, even if higher-order information carried by the faces is genuinely new spectral data, 
it does not broaden the class of Archimedean tilings that support flat bands. 
To show that this coincidence is not automatic --- that faces need not echo the vertices --- 
we examine the Lieb lattice, outside the Archimedean family. For this lattice, it is well-known that the vertex Laplacian 
has a flat band, but our investigations show that the Laplacian on faces does not. 
In this specific case, the operator $D_m$ inherits the vertex flat band but nothing more.

Let us finally provide the structure of this paper. 
Section \ref{sec:defdirac} sets up the framework: 1) planar graphs of bounded geometry, 
2) coboundary maps and the associated Hodge Laplacians on vertices, edges, and faces, 
3) Gauss--Bonnet and Dirac--Hodge operators $G_m$ and $D_m$. 
Section \ref{sec:spectral} develops Floquet--Bloch analysis of these operators 
on periodic graphs and establishes the spectral reduction expressing $\sigma(D_m)$ and $\sigma(G_m)$ 
in terms of the vertex and face Laplacians. 
The spurious flat band of $G_m$ and the resulting comparison between the two first-order operators
are also provided in this section. 
Section \ref{sec:flat} deals with the Dirac--Hodge operator 
on Archimedean lattices. We first treat the Kagome and the super-Kagome lattices, 
and then establish the absence of flat bands for the remaining Archimedean tilings. 
The Lieb lattice is then investigated as a contrast. 
In Section \ref{sec_loc}, we provide information about localised eigenvectors
for the Dirac--Hodge operator for the Kagome and super-Kagome lattices.
These results have been triggered by an interesting question from the referee.
In a short appendix, we also recall the duality between the face Laplacian of a planar graph and the vertex Laplacian 
of its dual graph.

\section{Planar graphs, Laplacians, and Dirac--Hodge operators}\label{sec:defdirac}

In this section we introduce the main objects for our investigations: planar graphs
with bounded geometry, and various operators acting on them.
These operators are all defined in terms of two maps $\dd_0$ and $\dd_1$
between Hilbert spaces related to the graphs.

\subsection{Planar graphs}\label{sub:planar_grpahs}\label{subsec_planar}

An unoriented graph $(\V,\E)$ consists of a set $\V$ of vertices and a set $\E$ 
of unoriented edges. For any edge $\e=\{\x,\y\}\in \E$, with $\x,\y\in \V$, we define the two 
oriented edges $\e=(\x,\y)$ and $\overline{\e}=(\y,\x)$, and denote by $\A$ the corresponding set of 
oriented edges. Thus, for any edge $\e\in \A$ the reverse edge $\overline{\e}$ also belongs to $\A$,
and conversely the two oriented edges $\e$ and $\overline{\e}$ are identified in $\E$. 
For an oriented edge $\e=(\x,\y)$ we write 
$o(\e)=\x$ and $t(\e)=\y$ for its origin and its target, respectively.  
Furthermore, the set of oriented edges coming out of $x\in \V$ is denoted by 
\begin{equation*}
\A_\x := \{ \e\in\A \mid o(\e)=\x \}
\end{equation*}
while the set of edges having one endpoint at $\x$ is denoted by $\E_\x$.

A cycle is an ordered collection of oriented edges
$(\e_1,\e_2,\dots,\e_n)$ such that $t(\e_j)=o(\e_{j+1})$ for $j=1,\dots,n-1$ and
$t(\e_n)=o(\e_1)$. The cycle is simple if $t(\e_i)\neq t(\e_j)$ whenever $i\neq j$ . 
A chordless cycle in a graph is a simple cycle such that no two vertices of the cycle are 
connected by an edge that does not itself belong to the cycle. 
A graph is said to be $2-$cell embedded into a surface if the interior of each chordless cycle 
is homeomorphic to an open disk. 
 
\begin{Definition}[Planar graph]\label{def:planar}
A \emph{planar graph} consists of an oriented graph $(\V,\A)$ endowed with a fixed $2-$cell
embedding into the plane. Its \emph{faces} are defined by chordless cycles, 
and the set of its faces is denoted by $\F$. 
\end{Definition}

In the sequel, we simply write $X=(\V,\A,\F)$ for any planar graph.
Observe that the notation contains the set of oriented edges in order to emphasise the role of orientation. 
For a face $\f=(\e_1,\dots,\e_n)$ consisting of a chordless cycle, we write $\overline{\f}\in \F$ for the chordless cycle
$(\overline{\e_n},\dots,\overline{\e_1})$ corresponding to the reverse direction.
Note that each oriented edge of the graph belongs to exactly two faces.

\begin{Remark}
The assumption of a $2-$cell embedding ensures that faces are well defined as 
$2-$dimensional objects whose boundaries correspond to polygonal simple closed curves in the plane. 
Although any embedding induces a rotation system on $X$ and thus a notion of faces, 
in general such faces need neither be finite nor closed, and therefore do not correspond to genuine $2$-cells. 
In particular, they would not support the homological framework required to develop 
the discrete differential calculus considered in this work. 
The cellularity hypothesis is thus both geometrically motivated by the class of 
locally finite periodic planar graphs studied below, and sufficient to develop a discrete Hodge theory.
\end{Remark}

As for oriented edges, we define the set of unoriented faces $\CC$ by identifying the faces $\f$ and  $\overline{\f}$. 
Thus, we call the triple $(\V,\E,\CC)$ the unoriented data for the planar graph $X$. 
In this setting as well, observe that any unoriented edge belongs to exactly two unoriented faces. 
We call such faces adjacent.
Note that $X=(\V,\A,\F)$ can be recovered from $(\V,\E,\CC)$. 
Indeed, the set $\A$ is obtained by setting for each edge in $\E$ two oriented edges, 
and subsequently these edges induce naturally the two orientations for $\F$. 

\subsection{Discrete differential operators and Laplacians}\label{sub:discrete_differential_operators}

Let $X=(\V,\A,\F)$ be a planar graph. We define the Hilbert space associated with vertices
$$
\ell^2(\V):=\Big\{f\colon \V\to \C \mid \sum_{\x\in \V} |f(\x)|^2 <\infty \Big\},
$$
endowed with the scalar product
\begin{equation*}
\langle f_1,f_2\rangle_{\ell^2(\V)}:=\sum_{\x\in \V} \overline{f_1(\x)}f_2(\x),\quad f_i\in \ell^2(\V).
\end{equation*}
Similarly, we define a second Hilbert space associated with oriented edges
\begin{equation*}
\ell^2(\A):=\Big\{g\colon \A\to \C \mid  g(\overline{\e})=-g(\e)\ \text{and}\
\sum_{\e\in \A}|g(\e)|^2 <\infty \Big\},
\end{equation*}
endowed with the scalar product
\begin{equation*}
\langle  g_1,g_2\rangle_{\ell^2(\A)}:= \frac{1}{2}\sum_{\e\in \A}\overline{g_1(\e)}g_2(\e),\quad g_i\in \ell^2(\A) .
\end{equation*}

A natural operator can be defined between these two Hilbert spaces. Namely,
let us set $\dd_0\colon \ell^2(\V)\to\ell^2(\A)$ defined on $f\in \ell^2(\V)$ and for $\e\in \A$ by
\begin{equation*}
[\dd_0f](\e)=f\big(t(\e)\big)-f\big(o(\e)\big).
\end{equation*}
For the boundedness of this operator, let $\deg_0\colon \V\to\N$ be given by $\deg_0(\x):=\sharp\A_\x$.
Assume then that 
$$
\deg_0(X):=\sup_{\x\in\V}\deg_0(\x)<\infty.
$$
Under this assumption, a straightforward computation shows that 
\begin{equation*}
\|\dd_0f\|^2_{\ell^2(\A)}\leq 2 \deg_0(X)\|f\|^2_{\ell^2(\V)},\quad\text{ for every }f\in\ell^2(\V).
\end{equation*}
Under the same assumption, the adjoint of $\dd_0$ is given on $g\in\ell^2(\A)$ and for $\x\in \V$ by
\begin{equation*}
[\dd_0^*g](\x)=-\sum_{\e\in\A_\x}g(\e).
\end{equation*}
Then the usual graph-Laplacian is obtained from the equalities
\begin{equation}\label{eq:Deltavertices}
[\dd_0^*\dd_0f](\x)=-\sum_{\e\in\A_\x}\Big(f\big(t(\e)\big)-f\big(o(\e)\big)\Big)=\deg_0(\x)f(\x)-\sum_{\e\in\A_\x}f\big(t(\e)\big)
\end{equation}
by setting 
$$
\Delta_{\V}:=-\dd_0^*\dd_0.
$$

Let us now turn our attention to the $2$-dimensional objects and define $\ell^2(\F)$ by
\begin{equation*}
\ell^2(\F):=\Big\{h\colon\F\to \C \mid  h(\overline{\f})=-h(\f)\ \text{and}\
\sum_{\f\in \F}|h(\f)|^2 <\infty \Big\},
\end{equation*}
endowed with the scalar product
\begin{equation*}
\langle  h_1,h_2\rangle_{\ell^2(\F)}:= \frac{1}{2}\sum_{\f\in \F}\overline{h_1(\f)}h_2(\f),\quad h_i\in \ell^2(\F).
\end{equation*}
We can then consider the coboundary operator $\dd_1\colon \ell^2(\A)\to\ell^2(\F)$ 
defined on $g\in \ell^2(\A)$ and for $\f\in \F$ by
\begin{equation*}
[\dd_1g](\f):= \sum_{\e\in \f}g(\e).
\end{equation*} 
As before, in order to ensure the boundedness of this operator we need a notion of boundedness of the geometry.
Thus, let $\deg_2\colon\F\to\N$ be given for $\f=(\e_1,\dots,\e_n)\in \F$ by
\begin{equation*}
\deg_2(\f)=n,
\end{equation*}
and assume that 
$$
\deg_2(X):=\sup_{\f\in\F}\deg_2(\f)<\infty.
$$
Then the boundedness of $\dd_1$ follows from  
\begin{align*}
\|\dd_1g\|^2_{\ell^2(\F)}= & \frac12\sum_{\f\in\F}|[\dd_1g](\f)|^2\\
\leq & \frac{1}{2}\sum_{\f\in\F}\deg_2(\f)\sum_{\e\in\f}|g(\e)|^2\\
\leq & \deg_2(X)\frac{1}{2}\sum_{\f\in\F}   \sum_{\e\in\f}|g(\e)|^2\\
= & \deg_2(X)\cdot 2\|g\|^2_{\ell^2(\A)}
\end{align*}
where the planarity of $X$ has been used in the last equality, \
implying that each oriented edge lies in only two oriented faces.
The adjoint $\dd_1^*\colon \ell^2(\F)\to \ell^2(\A)$ is then 
given on $h\in \ell^2(\F)$ and for $\e\in \A$ by 
\begin{equation*}
[\dd_1^*h](\e)=\sum_{\f\in\F\mid \e\in \f}h(\f).
\end{equation*}
To define the Laplacian on faces, we compute,
\begin{equation}\label{eq:Deltafaces}
[\dd_1\dd_1^*h](\f)= \sum_{\e \in \f}\sum_{\f'\in\F|\e\in \f'} h(\f')
=\deg_2(\f)h(\f)-\sum_{\f'\sim \f }h(\f'),
\end{equation}
where $\f'\sim \f$ if there exists $\e \in \f$ such that $\overline{\e}\in \f'$. 
Note that for the second equality, we have used that $\deg_2(\f)=\#\{\f'\sim \f\}$,
another consequence of the planarity of $X$. 
Finally, we can set 
$$
\Delta_{\F}:=-\dd_1\dd_1^*
$$
for the Laplacian on faces. 

\begin{Remark}
Graphs satisfying the two conditions $\deg_0(X)<\infty$ and $\deg_2(X)<\infty$ will be referred
as graphs with bounded geometry. For them, the operators introduced so far are bounded.
We refer to \cite{Ke15} for the notion of unbounded geometry for graphs.
\end{Remark}

Let us finally turn our attention to the Laplacian on edges. 
Since edges can be related to each other via either vertices or faces, the Laplacian on edges encompasses both interactions by
\begin{equation}\label{eq:Laplacianedges}
\Delta_{\A}:=-\dd_0\dd_0^*-\dd_1^*\dd_1\equiv \Delta^{\text{down}}_{\A}+\Delta^{\text{up}}_{\A}.
\end{equation} 
In Section \ref{sub:spectral}, this operator will be further analysed based 
on the \emph{Hodge decomposition}
\begin{equation}\label{eq:Hodgedecomposition}
\ell^2(\A)=\Ran(\dd_0)\oplus \Ker(\Delta_\A) \oplus \Ran(\dd_1^*).
\end{equation}
This equality has been proved in a more general framework in \cite[Sec.~2.1]{MP23}
and can be checked by explicit computations.
Thus, we have defined so far three Laplace-type operators, acting respectively on vertices, edges, and faces.
By considering graphs with bounded geometry, these operators are all bounded.

\subsection{Gauss--Bonnet and Dirac--Hodge operators}\label{sub:DH_operators}

Let us come to Dirac-type operators, based on the operators $\dd_0$ and $\dd_1$ 
introduced in the previous section. As developed in \cite{AT15,Pa17}, 
if one considers graphs only as $1-$dimensional objects, then it is natural 
to define the so-called \emph{Gauss--Bonnet} operator on $\H_{GB}(X)=\ell^2(\V)\oplus \ell^2(\A)$ by 
\begin{equation*}
G:=\begin{pmatrix}
0 & \dd_0^*\\
\dd_0 & 0
\end{pmatrix}.
\end{equation*}
Since it satisfies 
\begin{equation*}
G^2=\begin{pmatrix}
-\Delta_\V & 0\\
0 & -\Delta^{\text{down}}_{\A}
\end{pmatrix}
\end{equation*}
the operator $G$ is a (discrete) Dirac-type operator. 
On the other hand, and as emphasized \cite{MP23,Bi24}, if one includes $2-$dimensional information in graphs,
one is led to the so-called \emph{Dirac--Hodge} operator acting on 
$\H_{DH}(X)=\ell^2(\V)\oplus \ell^2(\A)\oplus \ell^2(\F)$ by
\begin{equation*}{}
D:=\begin{pmatrix}
0 & \dd_0^* & 0\\
\dd_0 & 0 & \dd_1^* \\
0 & \dd_1 & 0
\end{pmatrix}.
\end{equation*}
This operator satisfies 
\begin{equation*}
D^2=\begin{pmatrix}
-\Delta_\V & 0 & 0\\
0 & -\Delta_\A & 0 \\
0 & 0 & -\Delta_\F
\end{pmatrix}.
\end{equation*}
Note that the distinction between the names Gauss--Bonnet and the Dirac--Hodge operators 
is purely contingent on the developments of the literature of discrete operators, 
since both operators are intended as discrete counterparts of the same family of continuous objects. 
We adopt the distinction here in order to make future comparisons easier to follow.

By analogy to continuous Dirac operators, we also introduce the two involutions
$\tau\colon \H_{GB}(X)\to \H_{GB}(X)$ and 
$\tau\colon \H_{DH}(X)\to\H_{DH}(X)$ defined respectively by
$\tau(f,g)=(f,-g)$ and $\tau(f,g,h)=(f,-g,h)$.
Then, for any fixed mass $m\geq 0$, we can set
\begin{equation}\label{eq:Gmdef}
G_m:=G+m\tau =\begin{pmatrix}
m & \dd_0^*\\
\dd_0 & -m
\end{pmatrix}
\end{equation}
and 
\begin{equation}\label{eq:Dmdef}
D_m:=D+m\tau=\begin{pmatrix}
m & \dd_0^* & 0\\
\dd_0 & -m & \dd_1^* \\
0 & \dd_1 & m
\end{pmatrix}.
\end{equation}

\section{Spectral theory for periodic planar graphs}\label{sec:spectral}

In this section, we consider a periodic planar graph, namely a planar graph $X$ 
with bounded geometry and which  
admits a cellular embedding compatible with an action of $\Z^2$ by graph automorphisms. 
Note that the graph automorphism also acts on faces by its action on edges. 
We denote by $\XX :=\big(\V(\XX),\A(\XX),\F(\XX)\big)$ the quotient graph $ X/\Z^2$, 
which inherits a finite graph structure from $X$, 
and denote by $\omega\colon X\to\XX$ the quotient map. 
The periodic planarity of $X$ implies that $\XX$ admits a $2-$cell embedding in the $2$-torus $\T^2$.
We denote by $\nv$ the number of vertices in $\V(\XX)$, by $\ne$ the number of unoriented edges in $\E(\XX)$, 
and by $\nf$ the number of unoriented faces in $\CC(\XX)$. 
Since the Euler characteristic of $\T^2$ is zero, 
we obtain that every periodic planar graph satisfies (see for instance \cite[Chp. 3]{MoTh01})
\begin{equation}\label{eq:euler0}
\nv - \ne + \nf = 0.
\end{equation}
Let us stress that in the quotient $\XX$, the boundaries of faces may traverse vertices and edges multiple times, 
as shown for example in Figure \ref{fig:kagome_and_quotient}.
Nevertheless, the resulting cellular decomposition of $\T^2$ yields the correct equality.

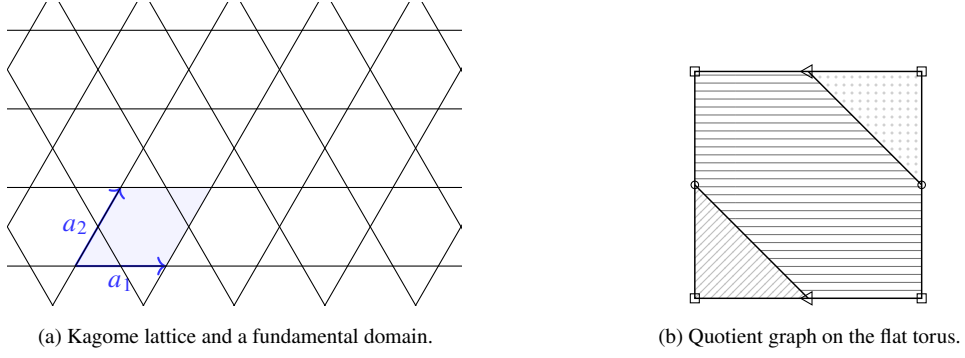
\begin{figure}[t]
\centering

\begin{minipage}[t]{0.58\textwidth}
  \centering
 \begin{tikzpicture}[scale = 2]
\begin{scope}[yshift = -3.5cm]
\clip (0,0) rectangle (3,2);
\coordinate (P) at (0.45,0.26);

\coordinate (aone) at (0.6,0);
\coordinate (atwo) at (0.3,{0.3*sqrt(3)});

\coordinate (P1)  at ($(P)+(aone)$);
\coordinate (P2)  at ($(P)+(atwo)$);
\coordinate (P12) at ($(P)+(aone)+(atwo)$);

\fill[blue!10, opacity=0.45] (P) -- (P1) -- (P12) -- (P2) -- cycle;

\draw[->, thick, blue!80] (P) -- (P1) node[midway, below] {$a_1$};
\draw[->, thick, blue!80] (P) -- (P2) node[midway, left] {$a_2$};

\begin{scope}[scale = .3]
  \foreach \i in {0,...,6}
    {
    \foreach \j in {0,...,3}
        {
        \pgfmathsetmacro{\ii}{2*\i}
        \pgfmathsetmacro{\jj}{4*\j*0.86602540378}
        \begin{scope}[xshift = \ii cm, yshift = \jj cm]
                \draw (0:1) -- (60:1) -- (120:1) -- (180:1) --
                      (240:1) -- (300:1) -- (360:1);
                
            \begin{scope}[xshift = 1cm, yshift = 2*0.86602540378cm]
                \draw (0:1) -- (60:1) -- (120:1) -- (180:1) --
                      (240:1) -- (300:1) -- (360:1);
            \end{scope}
        \end{scope}
        }
    }
\end{scope}
\end{scope}

 \end{tikzpicture}

  \vspace{2pt}
  {\footnotesize (a) Kagome lattice and a fundamental domain.}
\end{minipage}\hfill
\begin{minipage}[t]{0.38\textwidth}
  \centering
\begin{tikzpicture}[scale=3,
  sqv/.style={draw, regular polygon, regular polygon sides=4, inner sep=1.2pt},
  midLR/.style={draw, circle, inner sep=1.0pt},
  midTB/.style={draw, regular polygon, regular polygon sides=3, inner sep=1.0pt, rotate=90},
  edge/.style={line width=0.55pt}
]

\coordinate (A) at (0,0);
\coordinate (B) at (1,0);
\coordinate (C) at (1,1);
\coordinate (D) at (0,1);

\coordinate (Mbot)   at (0.5,0);
\coordinate (Mtop)   at (0.5,1);
\coordinate (Mleft)  at (0,0.5);
\coordinate (Mright) at (1,0.5);

\path[pattern=north east lines, pattern color=black!25, 
]
  (A) -- (Mbot) -- (Mleft) -- cycle;

\path[pattern=horizontal lines, pattern color=black!54, 
]
  (Mbot) -- (B) -- (Mright) -- (Mtop) -- (D) -- (Mleft) -- cycle;

\path[pattern=dots, pattern color=black!18, 
]
  (C) -- (Mtop) -- (Mright) -- cycle;

\draw[edge] (A)--(B)--(C)--(D)--cycle;
\draw[edge] (Mleft) -- (Mbot);
\draw[edge] (Mtop) -- (Mright);

\node[sqv] at (A) {};
\node[sqv] at (B) {};
\node[sqv] at (C) {};
\node[sqv] at (D) {};

\node[midTB] at (Mbot)   {};
\node[midTB] at (Mtop)   {};
\node[midLR] at (Mleft)  {};
\node[midLR] at (Mright) {};

\end{tikzpicture}

  \vspace{2pt}
  {\footnotesize (b) Quotient graph on the flat torus.}
\end{minipage}

\caption{(a) Kagome lattice with a fundamental domain and the choice of generators for the $\Z^2$-action, denoted by $a_1$ and $a_2$. Figure adapted from \cite{PeTa21}.
(b) The induced quotient graph embedded in the flat torus model of $\T^2$. In this case, we have $\nv=3$, $\ne=6$, and $\nf=3$.}
\label{fig:kagome_and_quotient}
\end{figure}

\subsection{Floquet--Bloch theory} \label{sub:bloch}

Floquet--Bloch theory for periodic graphs is a well developed subject, we present here just what is necessary for our aim. 
The crucial observation is that both $\dd_0$ and $\dd_1$ commute with the action of $\Z^2$. 
Indeed, for $\mu\in \Z^2$, if we denote by $\tau_\mu$ the translation on forms and simply by $\mu$ the action
of translations on $\x$, $\e$, and $\f$, then we have  
\begin{equation*}
[\dd_0\tau_\mu f](\e)=\big[\tau_\mu f\big]\big(t(\e)\big)-\big[\tau_\mu f\big] \big(o(\e)\big)
=f\big(\mu t(\e)\big)-f\big(\mu o(\e)\big)=[\dd_0f](\mu \e)=\big[\tau_\mu(\dd_0 f)\big](\e),
\end{equation*}
and analogously
 \begin{equation*}
[\dd_1\tau_\mu g](\f)=\sum_{\e\in\f}[\tau_\mu g](\e)=\sum_{\e\in\f}g(\mu\e)=\sum_{\e\in\mu\f}g(\e)=[\dd_1g](\mu \f)
= \big[\tau_\mu(\dd_1 g)\big](\f).
\end{equation*}

For the operator $G_m$ introduced in \eqref{eq:Gmdef}, it follows from \cite[Prop. 3.5]{Pa17} that 
there exists a unitary transformation $\U\colon\H_{GB}(X)\to L^2(\T^2,\C^{\nv+\ne})$ such that 
$\U G_m \U^*$ can be understood as a multiplication operator by a matrix 
$G_m(\xi)\in M_{\ne\times \nv}(\C)$. 
Our aim is to show how to extend $\U$ to $\H_{DH}(X)$ and to describe 
the operator $\U D_m \U^*$ with $D_m$ introduced in \eqref{eq:Dmdef}.
To construct $\U$, one needs to fix a set of representatives of $\XX$ in $X$. 
Thus, let us set
\begin{align*}
\V(\XX)&=(\xx_1,\xx_2,\dots,\xx_\nv),\\
\A(\XX)&=(\ee_1,\ee_2,\dots,\ee_\ne,\overline{\ee_1},\dots,\overline{\ee_\ne}),\\
\F(\XX)&=(\ff_1,\ff_2,\dots,\ff_\nf,\overline{\ff_1},\dots,\overline{\ff_\nf}).
\end{align*}
We also fix $\V_0\subset \V$, $\A_0\subset \A$, and $\F_0\subset \F$, such that they correspond 
one-to-one with $\V(\XX)$, $\A(\XX)$, and $\F(\XX)$ respectively. 
More precisely, we firstly choose a set of $\nv$ vertices $\hx_i\in \V$ for $\V_0$ such that $\omega(\hx_i)=\xx_i$
for all $i\in \{1,\dots,\nv\}$. 
Then, we choose a set of $\ne$ oriented edges $\he_i\in \A$ such that $\omega(\he_i)=\ee_i$ for all 
$i\in \{1,\dots,\ne\}$.
For simplicity and without loss of generality, we impose that $o(\he_i)\in\V_0$. 
We define the set $\A_0:=\{\he_i, \overline{\he_i}\}_{i=1}^\ne$ which is in bijection with $\A(\XX)$.
Finally, we choose a set of $\nf$ oriented chordless cycles $\hf_i$ in $\F$ such that 
$\omega(\hf_i)=\ff_i$ for all $i\in \{1,\dots,\nf\}$.
The set $\F_0$ is defined by $\{\hf_i, \overline{\hf_i}\}_{i=1}^\nf$ which is in bijection with $\F(\XX)$.
Thus, with these notations and conventions we have defined a bijective map 
$$
\widehat{\cdot}\ \colon \big(\V(\XX),\A(\XX),\F(\XX) \big) \to \big(\V_0,\A_0,\F_0 \big)
$$
that we call the lift map.
Based on this map, one can define a suitable coordinate system on $X$:
for every vertex $\x$, edge $\e$, or face $\f$ in $X$ we define 
$\lfloor \x \rfloor$, $\lfloor \e \rfloor$, $\lfloor \f \rfloor$ in $\Z^2$, by the relations
\begin{equation*}
\x=\lfloor \x \rfloor\widehat{\omega(\x)}\,, \qquad \e=\lfloor \e \rfloor\widehat{\omega(\e)}\,, 
\qquad \f=\lfloor \f \rfloor\widehat{\omega(\f)}\,.
\end{equation*}

We can now define $\U\colon \H_{DH}(X)\to L^2(\T^2,\C^{\nv+\ne+\nf})$ 
given for $(f,g,h) \in \H_{DH}$ and $i\in \{1,\dots, \nv+\ne+\nf\}$ by
\begin{equation*}
[\U(f,g,h)]_i(\xi):=\begin{cases}
\sum_{\mu\in\Z^2}e^{-2\pi i \xi \cdot \mu} f(\mu \hx_i) &\text{ if  }1\leq i\leq \nv,\\
\sum_{\mu\in\Z^2}e^{-2\pi i \xi \cdot \mu} g(\mu \he_{i-\nv}) &\text{ if  }\nv+1\leq i\leq \nv+\ne,\\
\sum_{\mu\in\Z^2}e^{-2\pi i \xi \cdot \mu} h(\mu \hf_{i-\nv-\ne}) &\text{ if  }\nv+\ne+1\leq i.
\end{cases}
\end{equation*}
Another convenient way to look at this map is to decompose it into three parts:
\begin{align*}
&\U_\V\colon \ell^2(\V) \to L^2(\T^2,\C^{\nv}), \quad [\U_\V f]_i(\xi) 
= \sum_{\mu\in\Z^2}e^{-2\pi i \xi \cdot \mu} f(\mu \hx_i) & 1\leq i\leq \nv,\\
&\U_\A\colon \ell^2(\A) \to L^2(\T^2,\C^{\ne}), \quad  [\U_\A g]_i(\xi) 
= \sum_{\mu\in\Z^2}e^{-2\pi i \xi \cdot \mu} g(\mu \he_{i}) & 1\leq i\leq \ne, \\
&\U_\F\colon \ell^2(\F) \to L^2(\T^2,\C^{\nf}), \quad [\U_\F h]_i(\xi) 
= \sum_{\mu\in\Z^2}e^{-2\pi i \xi \cdot \mu} h(\mu \hf_{i}) & 1\leq i\leq \nf.\\
\end{align*}
Note that the indexations in the final spaces are slightly simpler since the vector space $\C^{\nv+\ne+\nf}$
has been decomposed into three distinct parts.
By considering these three maps, we can now study the natural block decomposition 
of the unitary conjugation $\U  D_m \U^*$, with $D_m$ given in \eqref{eq:Dmdef}, namely
\begin{equation*}
\U\begin{pmatrix}
m & \dd_0^* & 0\\
\dd_0 & -m & \dd_1^* \\
0 & \dd_1 & m
\end{pmatrix} \U^* = 
\begin{pmatrix}
m & \U_\V\dd_0^* \U_\A^*& 0\\
\U_\A \dd_0 \U_\V^*& -m &\U_\A \dd_1^* \U_\F^* \\
0 & \U_\F\dd_1 \U_\A^*& m
\end{pmatrix}.
\end{equation*}
The following two statements provide the expressions appearing in this matrix.
Before its statement, one has to observe that any $\hf_i\in \F_0$ can be expressed
in terms of elements of $\A_0$. More precisely, given $\hf_i=(\e_1,\dots,\e_n)\in \F_0$,
each $\e_k$ can be written uniquely as $\lfloor \e_k\rfloor \he_j$ or as $\lfloor \e_k\rfloor \overline{\he_j}$
with $j$ defined by $\omega(\e_k)\in \{\ee_j,\overline{\ee_j}\}$.

\begin{Proposition}\label{prop_d1xi}
Let $X=(\V,\A,\F)$ be a periodic planar graph with bounded geometry. 
Then $\U_\F \dd_1\U_\A^*$ is the multiplication operator from $L^2(\T^2,\C^\ne)$ to $L^2(\T^2,\C^\nf)$ 
by the real-analytic matrix-valued function $\dd_1(\cdot):\T^2\to M_{\nf\times \ne}(\C)$ defined for $\xi \in \T^2$ by
$$
\dd_1(\xi)_{ij} = 
\sum_{k\mid \omega(\e_k)=\ee_j} e^{2\pi i\xi \cdot \lfloor \e_k\rfloor}
- \sum_{k\mid \omega(\e_k)=\overline{\ee_j}} e^{2\pi i\xi \cdot \lfloor \e_k\rfloor}.
$$
\end{Proposition}

\begin{proof}
Let $\hf_i=(\e_1,\dots,\e_n)\in \F_0$. Then for $\alpha\in L^2(\T^2,\C^\ne)$ one has
\begin{align*}
&\big[\U_\F \dd_1\U_\A^*\alpha\big]_i(\xi) \\ 
& = \sum_{\mu\in \Z^2}e^{-2\pi i\xi \cdot \mu}\big[\dd_1\U_\A^*\alpha\big](\mu \hf_i) \\
& = \sum_{\mu\in \Z^2}e^{-2\pi i\xi \cdot \mu} \sum_{\e\in \mu \hf_i}\big[\U_\A^*\alpha\big](\e) \\
& = \sum_{\mu\in \Z^2}e^{-2\pi i\xi \cdot \mu} \sum_{k=1}^n\big[\U_\A^*\alpha\big](\mu \e_k) \\
& = \sum_{\mu\in \Z^2}e^{-2\pi i\xi \cdot \mu} \bigg( \sum_{k\mid \omega(\e_k)\in \{\ee_j\}_j}\big[\U_\A^*\alpha\big]\big((\mu+\lfloor \e_k\rfloor) \widehat{\omega(\e_k)}\big)
- \sum_{k\mid \omega(\overline{\e_k})\in \{\ee_j\}_j}\big[\U_\A^*\alpha\big]\big((\mu+\lfloor \e_k\rfloor) \widehat{\omega(\overline{\e_k})}\big) \bigg) \\
& = \sum_{k\mid \omega(\e_k)\in \{\ee_j\}_j}e^{2\pi i\xi \cdot \lfloor \e_k\rfloor} 
\sum_{\mu\in \Z^2}e^{-2\pi i\xi \cdot (\mu+\lfloor \e_k\rfloor)} \big[\U_\A^*\alpha\big]\big((\mu+\lfloor \e_k\rfloor) \widehat{\omega(\e_k)}\big) \\
& \quad - \sum_{k\mid \omega(\overline{\e_k})\in \{\ee_j\}_j}e^{2\pi i\xi \cdot \lfloor \e_k\rfloor} 
\sum_{\mu\in \Z^2}e^{-2\pi i\xi \cdot (\mu+\lfloor \e_k\rfloor)} \big[\U_\A^*\alpha\big]\big((\mu+\lfloor \e_k\rfloor) \widehat{\omega(\overline{\e_k})}\big) \\
& =\sum_{j=1}^\ne  \sum_{k\mid \omega(\e_k)=\ee_j} e^{2\pi i\xi \cdot \lfloor \e_k\rfloor} 
\sum_{\mu\in \Z^2}e^{-2\pi i\xi \cdot (\mu+\lfloor \e_k\rfloor)} \big[\U_\A^*\alpha\big]\big((\mu+\lfloor \e_k\rfloor) \he_j\big) \\
& \quad - \sum_{j=1}^\ne  \sum_{k\mid \omega(\e_k)=\overline{\ee_j}} e^{2\pi i\xi \cdot \lfloor \e_k\rfloor} 
\sum_{\mu\in \Z^2}e^{-2\pi i\xi \cdot (\mu+\lfloor \e_k\rfloor)} \big[\U_\A^*\alpha\big]\big((\mu+\lfloor \e_k\rfloor) \he_j\big) \\
& = \sum_{j=1}^\ne  \sum_{k\mid \omega(\e_k)=\ee_j} e^{2\pi i\xi \cdot \lfloor \e_k\rfloor} [\alpha]_j(\xi)
- \sum_{j=1}^\ne  \sum_{k\mid \omega(\e_k)=\overline{\ee_j}} e^{2\pi i\xi \cdot \lfloor \e_k\rfloor} [\alpha]_j(\xi).
\end{align*}
One infers from the last expression that
$$
\dd_1(\xi)_{ij} = 
\sum_{k\mid \omega(\e_k)=\ee_j} e^{2\pi i\xi \cdot \lfloor \e_k\rfloor}
- \sum_{k\mid \omega(\e_k)=\overline{\ee_j}} e^{2\pi i\xi \cdot \lfloor \e_k\rfloor},
$$
as stated.
\end{proof}

For completeness, we adapt here the statement of \cite[Prop. 3.5]{Pa17} to our notations. 
For this, we set for any $e\in \A$
$$
\eta(\e):=\lfloor t(\e)\rfloor-\lfloor o(\e)\rfloor \in \Z^2.
$$
Note that the function $\eta\colon \A\to \Z^2$ is periodic, which implies that $\eta(\e)=\eta(\e')$ if 
$\omega(\e)=\omega(\e')$. Accordingly, we set $\eta(\ee):=\eta(\e)$ if $\omega(\e)=\ee\in \A(\XX)$.

\begin{Proposition}\label{prop_d2xi}
Let $X=(\V,\A,\F)$ be a periodic planar graph with bounded geometry. 
Then $\U_\A \dd_0\U_\V^*$ is the multiplication operator from $L^2(\T^2,\C^\nv)$ to $L^2(\T^2,\C^\ne)$ 
by the real-analytic matrix-valued function $\dd_0(\cdot)\colon \T^2\to M_{\ne\times \nv}(\C)$
defined for $\xi \in \T^2$ by
\begin{equation}\label{eq:d0xi}
    \dd_0(\xi)_{ij}=\begin{cases}
        0 &\text{ if } o(\ee_i)\neq \xx_j \neq t(\ee_i),\\
        -1 &\text{ if } o(\ee_i)= \xx_j \neq t(\ee_i),\\
        e^{2\pi i \xi\cdot\eta (\ee_i) } &\text{ if } o(\ee_i)\neq \xx_j = t(\ee_i),\\
        -1 + e^{2\pi i \xi\cdot\eta (\ee_i) } &\text{ if } o(\ee_i)= \xx_j = t(\ee_i).
    \end{cases}
\end{equation}  
\end{Proposition}

\begin{figure}[ht]
\begin{tikzpicture}[scale=8,
  vert/.style={circle, fill=black, inner sep=1.8pt},
  ghost/.style={circle, fill=gray!55, inner sep=1.4pt},
  earrow/.style={
    decoration={markings,
      mark=at position 0.55 with
        {\arrow[black]{Stealth[length=5pt,width=4pt]}}},
    postaction=decorate, line width=1.0pt},
  darrow/.style={
    decoration={markings,
      mark=at position 0.55 with
        {\arrow[gray!60]{Stealth[length=4pt,width=3.5pt]}}},
    postaction=decorate, line width=0.8pt, dashed, gray!60},
  xarrow/.style={
    decoration={markings,
      mark=at position 0.55 with
        {\arrow[blue!55]{Stealth[length=4pt,width=3.5pt]}}},
    postaction=decorate, line width=0.8pt, dashed, blue!55},
]
 
  \coordinate (Z1) at (0.45, 0.26);
  \coordinate (Z2) at (0.75, 0.26);
  \coordinate (Z3) at (0.60, 0.52);
 
  \coordinate (Z1a1)     at (1.05, 0.26);
  \coordinate (Z2ma1)    at (0.15, 0.26);   
  \coordinate (Z3ma2)    at (0.90, 0.00);   
  \coordinate (Z3Lma2)   at (0.30, 0.00);   
 
  \coordinate (Za2ma1x2) at (0.45, 0.78);   
  \coordinate (Za2ma1x1) at (0.15, 0.78);   
  \coordinate (Zma1x3)   at (0.00, 0.52);   
 
  \clip (-0.20,-0.15) rectangle (1.28,0.95);
 
  \begin{scope}[opacity=0.15, scale=.3]
    \foreach \i in {-1,...,5}{
      \foreach \j in {-1,...,5}{
        \pgfmathsetmacro{\ii}{2*\i}
        \pgfmathsetmacro{\jj}{4*\j*0.86602540378}
        \begin{scope}[xshift=\ii cm,yshift=\jj cm]
          \draw(0:1)--(60:1)--(120:1)--(180:1)--(240:1)--(300:1)--(360:1);
          \begin{scope}[xshift=1cm,yshift={2*0.86602540378cm}]
            \draw(0:1)--(60:1)--(120:1)--(180:1)--(240:1)--(300:1)--(360:1);
          \end{scope}
        \end{scope}
      }
    }
  \end{scope}
 
  \draw[earrow](Z1)--(Z2)
    node[midway, below=5pt, scale=1.10]{$\he_1$};
  \draw[earrow](Z1)--(Z2ma1)
    node[midway, above=4pt, scale=1.10]{$\he_2$};
  \draw[earrow](Z1)--(Z3)
    node[midway, left=5pt, scale=1.10]{$\he_3$};
  \draw[earrow](Z1)--(Z3Lma2)
    node[pos=0.55, right=4pt, scale=1.10]{$\he_4$};
  \draw[earrow](Z2)--(Z3)
    node[midway, right=5pt, scale=1.10]{$\he_5$};
  \draw[earrow](Z2)--(Z3ma2)
    node[pos=0.55, right=4pt, scale=1.10]{$\he_6$};
 
  \draw[xarrow](Z3)--(Za2ma1x2)
    node[midway, right=4pt, scale=1.05]
      {$(a_2{-}a_1)\overline{\he_6}$};
  \draw[xarrow](Za2ma1x2)--(Za2ma1x1)
    node[midway, above=4pt, scale=1.05]
      {$(a_2{-}a_1)\overline{\he_1}$};
  \draw[xarrow](Za2ma1x1)--(Zma1x3)
    node[midway, left=4pt, scale=1.05]
      {$(a_2{-}a_1)\he_4$};
  \draw[xarrow](Zma1x3)--(Z2ma1)
    node[midway, left=4pt, scale=1.05]
      {$({-}a_1)\overline{\he_5}$};
  \draw[xarrow](Z3ma2)--(Z1a1)
    node[midway, right=4pt, scale=1.05]
      {$(a_1)\overline{\he_4}$};
  \draw[xarrow](Z1a1)--(Z2)
    node[midway, above=4pt, scale=1.05]
      {$(a_1)\he_2$};
 
  \node[ghost] at (Z2ma1)    {};
  \node[ghost] at (Z1a1)     {};
  \node[ghost] at (Z3ma2)    {};
  \node[ghost] at (Z3Lma2)   {};
  \node[ghost] at (Za2ma1x2) {};
  \node[ghost] at (Za2ma1x1) {};
  \node[ghost] at (Zma1x3)   {};
 
  \node[scale=1.00, gray!70, below=4pt]  at (Z3Lma2)
    {$(-a_2)\hx_3$};
  \node[scale=1.00, gray!70, below=4pt]  at (Z3ma2)
    {$(a_1-a_2)\hx_3$};
  \node[scale=1.00, gray!70, right=3pt]  at (Z1a1)
    {$(a_1)\hx_1{+}$};
  \node[scale=1.00, blue!60, right=3pt]  at (Za2ma1x2)
    {$(a_2{-}a_1)\hx_2$};
  \node[scale=1.00, blue!60, left=3pt]   at (Za2ma1x1)
    {$(a_2{-}a_1)\hx_1$};
  \node[scale=1.00, blue!60, left=3pt]   at (Zma1x3)
    {$({-}a_1)\hx_3$};
  \node[scale=1.00, blue!60, left=3pt]   at (Z2ma1)
    {$({-}a_1)\hx_2$};
 
  \node[vert] at (Z1) {};
  \node[vert] at (Z2) {};
  \node[vert] at (Z3) {};
 
  \node[scale=1.20, below left=3pt]  at (Z1){$\hx_1$};
  \node[scale=1.20, below right=3pt] at (Z2){$\hx_2$};
  \node[scale=1.20, above right=3pt] at (Z3){$\hx_3$};
 
  \node[scale=1.40, blue!70] at (0.9000, 0.1734){$\hf_3$};
  \node[scale=1.40, blue!70] at (0.3000, 0.5198){$\hf_1$};
  \node[scale=1.40, blue!70] at (0.6000, 0.3466){$\hf_2$};
 
\end{tikzpicture}
\caption{A particular choice of $\A_0,\V_0,\F_0$ for the Kagome lattice of Figure \ref{fig:kagome_and_quotient}. 
All the $\hf_j$ are considered with the counterclockwise orientation. The generators for the action of $\Z^2$
are indicated in Figure \ref{fig:kagome_and_quotient}(a).}
    \label{fig:fundamentalkagome}
\end{figure}
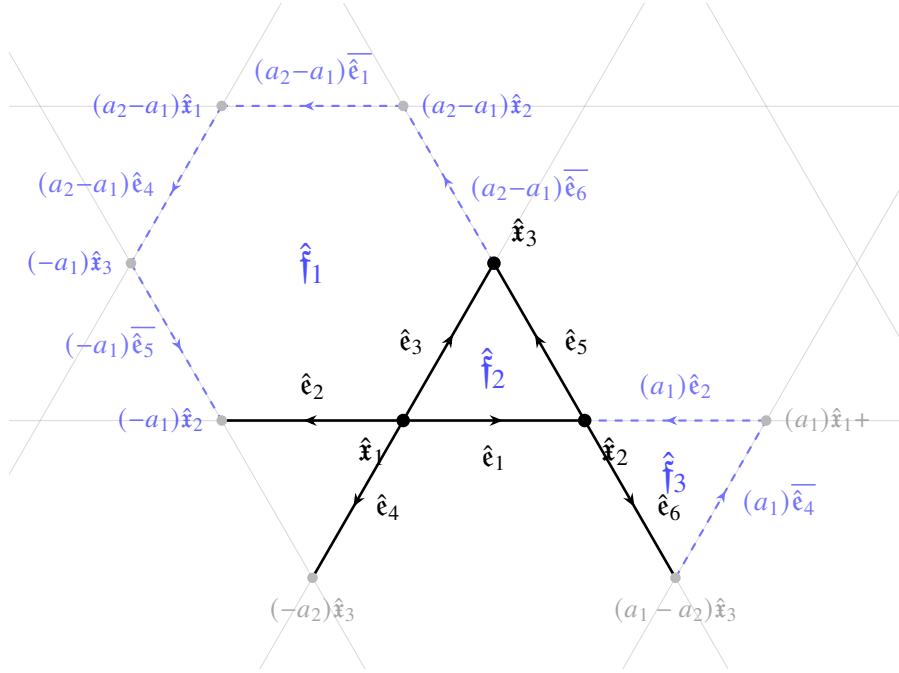
\begin{Example}
    Let us now consider the example of the Kagome lattice given in Figure \ref{fig:kagome_and_quotient} with the help of Figure \ref{fig:fundamentalkagome}. First, consider $\dd_0(\xi)\in M_{6\times 3}(\C)$ where each $i$-th line 
    gives the incident information for $\ee_i$:
    
    \begin{equation}
\dd_0(\xi)=\begin{pmatrix}
-1&1&0\\
-1&e^{-2\pi i\xi_1}&0\\
-1&0&1\\
-1&0&e^{-2\pi i\xi_2}\\
0&-1&1\\
0&-1&e^{2\pi i(\xi_1-\xi_2)}
\end{pmatrix}\ .\label{eq:d0Kagome}
\end{equation}
Moreover, $\dd_1(\xi)\in M_{3\times 6}(\C)$ where each $i$-th line gives the incident information for $\ff_i$ in terms of oriented edges:
\begin{equation}
\dd_1(\xi)=\begin{pmatrix}
-e^{2\pi i(\xi_2-\xi_1)}& -1 & 1 & e^{2\pi i(\xi_2-\xi_1)} & -e^{-2\pi i\xi_1} & -e^{2\pi i(\xi_2-\xi_1)} \\
1&0 & -1 & 0 & 1&0\\
0&e^{2\pi i\xi_1} & 0 & -e^{2\pi i\xi_1} & 0 & 1 
\end{pmatrix}\label{eq:d1Kagome}
    \end{equation}
\end{Example}

\subsection{Spectral reduction} \label{sub:spectral}

Our next aim is to study the spectrum of the operators $D_m$ and $G_m$ by exploiting their particular structure presented in \eqref{eq:Gmdef} and in \eqref{eq:Dmdef}. 
First of all, by Propositions \ref{prop_d1xi} and \ref{prop_d2xi} the operator $D_m$ is unitarily equivalent
to the matrix-valued function 
$$
\T^2\ni \xi \mapsto D_m(\xi):=\begin{pmatrix}
m \, I_\nv & \dd_0(\xi)^{*} & 0\\
\dd_0(\xi) & -m \, I_\ne &\dd_1(\xi)^* \\
0 & \dd_1(\xi) & m \, I_\nf
\end{pmatrix} \in M_{\nv+\ne+\nf}(\C).
$$
The first result is about the pointwise determinant of this function. 

\begin{Lemma}\label{lem_k3}
For $\xi \in \T^2$, the characteristic polynomial of $D_m(\xi)$ satisfies
\begin{equation}\label{eq_cara_2}
\det\big(D_m(\xi)-\lambda I_{\nv+\ne+\nf} \big) 
= \det \Big((\lambda^{2}-m^2)I_\ne - \dd_0(\xi)\dd_0(\xi)^*-\dd_1(\xi)^*\dd_1(\xi)\Big). 
\end{equation}
\end{Lemma}

\begin{proof}
By the block determinant formula, one gets for $\lambda \neq m$
\begin{align*}
&\det\big(D_m(\xi)-\lambda I_{\nv+\ne+\nf} \big)\\
& =\det \begin{pmatrix}
(m-\lambda) \, I_\nv & \dd_0(\xi)^{*} & 0\\
\dd_0(\xi) & (-m-\lambda) \, I_\ne &\dd_1(\xi)^* \\
0 & \dd_1(\xi) & (m-\lambda) \, I_\nf
\end{pmatrix} \\
& = \det \big((m-\lambda) \, I_\nf\big)  \det\left(\begin{pmatrix}
(m-\lambda) \, I_\nv & \dd_0(\xi)^{*} \\
\dd_0(\xi) & (-m-\lambda) \, I_\ne
\end{pmatrix}
- (m-\lambda)^{-1} \begin{pmatrix}
0 & 0\\
0 &\dd_1(\xi)^*\dd_1(\xi)
\end{pmatrix}\right) \\
& = (m-\lambda)^\nf
\ \det\begin{pmatrix}
(m-\lambda) \, I_\nv & \dd_0(\xi)^{*} \\
\dd_0(\xi) & (-m-\lambda) \, I_\ne - (m-\lambda)^{-1}\dd_1(\xi)^*\dd_1(\xi)
\end{pmatrix} \\
& = (m-\lambda)^{\nf+\nv}
\det \big(
(-m-\lambda) \, I_\ne - (m-\lambda)^{-1}\dd_1(\xi)^*\dd_1(\xi)
- (m-\lambda)^{-1}\dd_0(\xi)\dd_0(\xi)^*\big) \\
& = (m-\lambda)^{\nf+\nv-\ne}\,\det \bigl((\lambda^{2}-m^2)I_\ne - \dd_0(\xi)\dd_0(\xi)^*-\dd_1(\xi)^*\dd_1(\xi)\bigr).
\end{align*} 
By taking \eqref{eq:euler0} into account, the first factor vanishes. 
Finally, since both sides are polynomials in $\lambda$, the formula also holds for $\lambda=m$ by continuity. 
\end{proof}

Then, in order to compute the spectrum of $D_m$ let us study the matrix 
\begin{equation*}
-\Delta_\A(\xi):= \dd_0(\xi)\dd_0(\xi)^*+\dd_1(\xi)^*\dd_1(\xi)
\end{equation*}
For that purpose, note that the periodicity of the operators $\dd_0$ and $\dd_1$ implies that the Hodge decomposition \eqref{eq:Hodgedecomposition} can also be considered fiberwise (by conjugation with the unitary transform $\U_\A$)
and reads
\begin{equation}\label{eq_decHf}
\C^\ne=\Ran\big(\dd_0(\xi)\big)\oplus \Ker\big(\Delta_\A(\xi)\big) \oplus  \Ran\big(\dd_1(\xi)^*\big).
\end{equation}
Then, by setting 
$$-\Delta_\V(\xi):=\dd_0(\xi)^*\dd_0(\xi) \quad \hbox{ and }\quad -\Delta_\F(\xi):=\dd_1(\xi)\dd_1(\xi)^*
$$ 
we obtain the following characterization of $\sigma\big(-\Delta_\A(\xi)\big)$.

\begin{Proposition}\label{prop:spectrumduality}
For any $\xi\in\T^2$ one has
\begin{equation}\label{eq:spectrumduality}
\sigma\big(-\Delta_\A(\xi)\big)=\sigma\big(-\Delta_\V(\xi)\big)\cup\sigma\big(-\Delta_\F(\xi)\big).
\end{equation}
\end{Proposition}

Before the proof, observe that $-\Delta_\A(\xi)$ is a $\ne\times \ne$ matrix, and therefore has $\ne$ eigenvalues
(multiplicity included). On the other hand,  $-\Delta_\V(\xi)$ and $-\Delta_\F(\xi)$ have respectively $\nv$ and $\nf$ eigenvalues.
Thus, the equality \eqref{eq:spectrumduality} is coherent, once the relation \eqref{eq:euler0} is taken into account.

\begin{proof}
The two key observations for this proof is that the operator
$\dd_0(\xi)\dd_0(\xi)^*$ acts non-trivially uniquely in the first entry of the direct sum in \eqref{eq_decHf}
while the operator $\dd_1(\xi)^*\dd_1(\xi)$ acts non-trivially only in the third entry of the direct sum. 
By the standard relation between the spectra of $BB^*$ and $B^*B$ for any matrix $B$, one then gets
\begin{align}\label{eq_with0}
\begin{split}
\sigma\big(-\Delta_\A(\xi)\big)\setminus \{0\}
=&\Big(\sigma\big(\dd_0(\xi)\dd_0(\xi)^*\big)\cup\sigma\big(\dd_1(\xi)^*\dd_1(\xi)\big)\Big)\setminus\{0\}\\
=&\Big(\sigma\big(\dd_0(\xi)^*\dd_0(\xi)\big)\cup\sigma\big(\dd_1(\xi)\dd_1(\xi)^*\big)\Big)\setminus\{0\}\\
=&\Big(\sigma\big(-\Delta_\V(\xi)\big)\cup\sigma\big(-\Delta_\F(\xi)\big)\Big)\setminus\{0\}.
\end{split}
\end{align}

Now, if we assume that $\Ker\big(-\Delta_\A(\xi)\big)=\{0\}$ we get by the decomposition \eqref{eq_decHf} that
\begin{equation*}
\nv=\rank\big(\dd_0(\xi)\big)=\rank \big(\dd_0(\xi)^*\dd_0(\xi)\big) = \rank\big(-\Delta_\V(\xi)\big)
\end{equation*}
and 
\begin{equation*}
\nf=\rank\big(\dd_1(\xi)^*\big)=\rank\big(\dd_1(\xi)\dd_1(\xi)^*\big) = \rank\big(-\Delta_\F(\xi)\big),
\end{equation*}
and hence the equalities in \eqref{eq_with0} hold without the need of subtracting $\{0 \}$. 

On the other hand, if $\Ker\big(-\Delta_\A(\xi)\big)\neq \{0\}$, we infer by a dimensional argument that
$\rank\big(\dd_0(\xi)\big)< \nv $ or $\rank\big(\dd_1(\xi)\big)^*<\nf$, or both if the degeneracy is of greater multiplicity. 
In both cases, it means that $0$ belongs to all sets (or union of sets) in \eqref{eq_with0},
and therefore these equalities again hold without the need of subtracting $\{0 \}$. This directly leads to the statement.
\end{proof}

By collecting the content of \eqref{eq_cara_2}, \eqref{eq:euler0} and \eqref{eq:spectrumduality} we get the following
description of the spectrum of $D_m$. Note that the last sentence in the statement
is obtained by analyticity of the map $\xi\mapsto D_m(\xi)$.

\begin{Theorem}\label{theo:Dm}
For any $m\geq 0$ one has 
\begin{equation*}
\sigma(D_m) =\bigcup_{\xi\in\T^2}\sigma\big(D_m(\xi)\big)
\end{equation*}
with
\begin{align*}
\sigma\big(D_m(\xi)\big)
&= \Big\{-\sqrt{\sigma\big(-\Delta_\V(\xi)\big)+m^2}\Big\}
\cup \Big\{\sqrt{\sigma\big(-\Delta_\V(\xi)\big)+m^2}\Big\}\\
&\quad \cup \Big\{-\sqrt{\sigma\big(-\Delta_\F(\xi)\big)+m^2}\Big\}
\cup \Big\{\sqrt{\sigma\big(-\Delta_\F(\xi)\big)+m^2}\Big\}\ .
\end{align*}
In addition, $D_m$ has a flat band if and only if $-\Delta_\V$ or $-\Delta_\F$ have one. 
\end{Theorem}

For comparison, by applying the same strategy, one gets the following spectral characterisation for $G_m$.

\begin{Theorem}
For $\xi \in \T^d$, let $G_m(\xi)\in M_{\nv+\ne}(\C)$ be defined by
\begin{equation*}
G_m(\xi) :=
\begin{pmatrix}
m \, I_\nv & \dd_0(\xi)^{*}\\
\dd_0(\xi) & -m \, I_\ne
\end{pmatrix}.
\end{equation*}
Then, for any $m\geq 0$ one has 
\begin{equation*}
\sigma\big(G_m(\xi)\big) = \{-m\} \cup 
\Big\{-\sqrt{\sigma\big(-\Delta_\V(\xi)\big)+m^2}\Big\}
\cup \Big\{\sqrt{\sigma\big(-\Delta_\V(\xi)\big)+m^2}\Big\}
\end{equation*}
and
\begin{equation*}
\sigma(G_m)=
\bigcup_{\xi \in \T^2}\sigma \big(G_m(\xi)\big)
= \{-m\}\cup \Big\{-\sqrt{\sigma(-\Delta_\V)+m^2}\Big\}
\cup \Big\{\sqrt{\sigma(-\Delta_\V)+m^2}\Big\},
\end{equation*}
where the eigenvalue $-m$ has multiplicity $\nf$.
\end{Theorem}

As a consequence of these investigations, one observes that for every periodic graph, 
the operator $G_m$ exhibits a \emph{spurious flat band}, namely a flat band at $-m$ not linked 
to the precise geometry of the graph. For that reason, in order to extend the analysis of \cite{PeTa21,KeTaWi24}
to Dirac-type operators, the correct definition is given by $D_m$, not by $G_m$.

\section{Flat bands for the Dirac--Hodge operator on Archimedean tilings}\label{sec:flat}

In this section, we exploit Theorem \ref{theo:Dm} to study whether new flat band phenomena appear 
in the Archimedean tilings when considering the Dirac-type operator $D_m$. 
We start with the Kagome and the Super-Kagome tilings, since their spectrum is poised to exhibit flat bands coming from the vertex Laplacian, and then consider the remaining Archimedean lattices. At the end of this section, we study the Lieb lattice to contrast the results obtained in the Archimedean case. For ease of reference, we collect in Figure \ref{fig:11_tilings} the 11 Archimedean lattices with their Cundy \& Rollett's notation. 

\begin{figure}[ht]
  \centering

 \begin{tikzpicture}
  
\begin{scope}
\draw (1.25,-.5) node {$(4^4)$};

 \clip (0,0) rectangle (2.5,2.5); 
\begin{scope}[scale =.33, xshift = .33cm, yshift = .33cm]
   \draw (-1, -1) grid (8,8);
\end{scope}
\end{scope}
  
\begin{scope}[xshift = 3cm]
\draw (1.25,-.5) node {$(3^6)$};
 \clip (0,0) rectangle (2.5,2.5);
\begin{scope}[scale =.35]
  \foreach \i in {0,...,6}
    {
    \foreach \j in {0,...,3}
        {
        \pgfmathsetmacro{\jj}{2*\j*0.86602540378}
        \draw (\i,\jj) -- (\i+1,\jj);
        \draw (\i,\jj) -- (\i+1, \jj+2*0.86602540378);
        \draw (\i+1,\jj) -- (\i, \jj + 2*0.86602540378);
        \draw (\i, \jj + 0.86602540378) -- (\i + 1,\jj + 0.86602540378);
        }
    }
\end{scope}
\end{scope}

\begin{scope}[xshift = 6cm]
\draw (1.25,-.5) node {$(6^3)$};
 \clip (0,0) rectangle (2.5,2.5);
\begin{scope}[scale = .4, xshift = -.05cm]
  \foreach \i in {0,...,6}
    {
    \foreach \j in {0,...,3}
        {
        \pgfmathsetmacro{\ii}{2*\i*0.86602540378}
        \pgfmathsetmacro{\jj}{3*\j}
        \begin{scope}[xshift = \ii cm, yshift = \jj cm]
                \draw (30:1) -- (90:1) -- (150:1) -- (210:1) -- (270:1) -- (330:1) -- (30:1);
            
            \begin{scope}[xshift = 0.86602540378 cm , yshift = 1.5 cm]
                \draw (30:1) -- (90:1) -- (150:1) -- (210:1) -- (270:1) -- (330:1) -- (30:1);
            \end{scope}
        \end{scope}
        }
    }
\end{scope}
\end{scope}

\begin{scope}[yshift = -3.5cm]
\draw (1.25,-.5) node {$(3.6)^2$};
 \clip (0,0) rectangle (2.5,2.5);
\begin{scope}[scale = .2]
  \foreach \i in {0,...,6}
    {
    \foreach \j in {0,...,3}
        {
        \pgfmathsetmacro{\ii}{2*\i}
        \pgfmathsetmacro{\jj}{4*\j*0.86602540378}
        \begin{scope}[xshift = \ii cm, yshift = \jj cm]
                \draw (0:1) -- (60:1) -- (120:1) -- (180:1) -- (240:1) -- (300:1) -- (360:1);
                
            \begin{scope}[xshift = 1cm, yshift = 2*0.86602540378cm]
                \draw (0:1) -- (60:1) -- (120:1) -- (180:1) -- (240:1) -- (300:1) -- (360:1);
            \end{scope}
        \end{scope}
        }
    }
\end{scope}
\end{scope}    
    
\begin{scope}[xshift = 3cm, yshift = -3.5cm]
\draw (1.25,-.5) node {$(3.12^2)$};
 \clip (0,0) rectangle (2.5,2.5);
\begin{scope}[scale = .22] 
   \foreach \i in {0,...,6}
    {
    \foreach \j in {0,...,3}
        {
        \pgfmathsetmacro{\ii}{\i * (2*0.86602540378 + 2)}
        \pgfmathsetmacro{\jj}{\j * (4 * 0.86602540378 + 3)}
        \begin{scope}[xshift = \ii cm, yshift = \jj cm]
        \draw (0,0) -- (1,0) -- (.5,0.86602540378) -- (0,0);
        \draw (.5,0.86602540378) -- (.5,1.86602540378);
        \draw (.5,1.86602540378) -- (1,1.86602540378 + 0.86602540378) -- (0,1.86602540378 + 0.86602540378) -- (.5,1.86602540378);
            \begin{scope}
            \draw (210:1) -- (0,0);
            \end{scope}
            \begin{scope}[xshift = 1cm]
            \draw (330:1) -- (0,0);
            \end{scope}
            \pgfmathsetmacro{\l}{2*0.86602540378 + 1}
            \begin{scope}[yshift = \l cm]
            \draw (150:1) -- (0,0);
            \end{scope}
            \begin{scope}[xshift=1cm, yshift = \l cm]
            \draw (30:1) -- (0,0);
            \end{scope}
        \end{scope}
        
        \pgfmathsetmacro{\iii}{\ii + (0.86602540378 + 1)}
        \pgfmathsetmacro{\jjj}{\jj + (2 * 0.86602540378 + 1.5)}
        
        \begin{scope}[xshift = \iii cm, yshift = \jjj cm]
        \draw (0,0) -- (1,0) -- (.5,0.86602540378) -- (0,0);
        \draw (.5,0.86602540378) -- (.5,1.86602540378);
        \draw (.5,1.86602540378) -- (1,1.86602540378 + 0.86602540378) -- (0,1.86602540378 + 0.86602540378) -- (.5,1.86602540378);
            \begin{scope}
            \draw (210:1) -- (0,0);
            \end{scope}
            \begin{scope}[xshift = 1cm]
            \draw (330:1) -- (0,0);
            \end{scope}
            \pgfmathsetmacro{\l}{2*0.86602540378 + 1}
            \begin{scope}[yshift = \l cm]
            \draw (150:1) -- (0,0);
            \end{scope}
            \begin{scope}[xshift=1cm, yshift = \l cm]
            \draw (30:1) -- (0,0);
            \end{scope}
        \end{scope}
        }
    }
\end{scope}
\end{scope}

\begin{scope}[xshift=6cm, yshift=-3.5cm]
\draw (1.25,-.5) node {$(3^3.4^2)$};
 \clip (0,0) rectangle (2.5,2.5);
\begin{scope}[scale = .4, xshift = -0.05cm]

  \foreach \i in {-1,...,6}
    {
    \foreach \j in {0,...,2}
        {
        \pgfmathsetmacro{\jj}{\j * 2 * (1+0.86602540378)}
        \begin{scope}[xshift = \i cm, yshift = \jj cm]
            \draw (0,0) -- (0,1) -- (1,1) -- (1,0) -- (0,0);
            \draw(0,1) -- (0.5,1+0.86602540378) -- (1,1);
             \pgfmathsetmacro{\l}{1+0.86602540378}
                \begin{scope}[xshift = .5cm, yshift = \l cm]
                    \draw (0,0) -- (0,1) -- (1,1) -- (1,0) -- (0,0);
                    \draw(0,1) -- (0.5,1+0.86602540378) -- (1,1);
                \end{scope}
        \end{scope}
        }
    }
\end{scope}
\end{scope}

\begin{scope}[xshift=9cm, yshift=-3.5cm]
\draw (1.25,-.5) node {$(4.8^2)$};
 \clip (0,0) rectangle (2.5,2.5);
\begin{scope}[scale = .2, xshift = .25cm]
 \foreach \i in {0,...,3}
    {
    \foreach \j in {0,...,3}
        {
        \pgfmathsetmacro{\ii}{\i * 2 * (1+0.70710678118)}
        \pgfmathsetmacro{\jj}{\j * 2 * (1+0.70710678118)}
        \begin{scope}[xshift = \ii cm, yshift = \jj cm]
            \draw (0,0) -- (0,1) -- (1,1) -- (1,0) -- (0,0);
            \draw (0,0) -- (-0.70710678118,-0.70710678118);
            \draw (1,0) -- (1.70710678118,-0.70710678118);
            \draw (0,1) -- (-0.70710678118,1.70710678118);
            \draw (1,1) -- (1.70710678118,1.70710678118);
            \pgfmathsetmacro{\l}{1+0.70710678118}
                \begin{scope}[xshift = \l cm, yshift = \l cm]
                    \draw (0,0) -- (0,1) -- (1,1) -- (1,0) -- (0,0);
                    \draw (0,0) -- (-0.70710678118,-0.70710678118);
                    \draw (1,0) -- (1.70710678118,-0.70710678118);
                    \draw (0,1) -- (-0.70710678118,1.70710678118);
                    \draw (1,1) -- (1.70710678118,1.70710678118);
                \end{scope}
        \end{scope}
        }
    }
\end{scope}
\end{scope}

\begin{scope}[yshift = -7cm]
\draw (1.25,-.5) node {$(3^2.4.3.4)$};
 \clip (0,0) rectangle (2.5,2.5);
\begin{scope}[scale = .4]
 \foreach \i in {0,...,3}
    {
    \foreach \j in {0,...,3}
        {
        \pgfmathsetmacro{\l}{1 + 2* 0.86602540378}
        \pgfmathsetmacro{\ii}{\i * \l}
        \pgfmathsetmacro{\jj}{\j * \l}
        \begin{scope}[xshift = \ii cm, yshift = \jj cm]
            \pgfmathsetmacro{\ll}{0.86602540378 + 0.5}
            \begin{scope}[xshift = 0cm, yshift = 0cm]
            \draw (0,0) -- (1,0) -- (.5,-0.86602540378) -- (0,0);
            \draw (0,0) -- (1,0) -- (.5,0.86602540378) -- (0,0);  
            \end{scope}
            \begin{scope}[xshift = .5cm, yshift = 0.86602540378 cm]
                \draw (0,0) -- (0,1) -- (-0.86602540378, .5) -- (0,0);
                \draw (0,0) -- (0,1) -- (0.86602540378,.5) -- (0,0);
            \end{scope}
            \begin{scope}[xshift = \ll cm, yshift = \ll cm]
            \draw (0,0) -- (1,0) -- (.5,-0.86602540378) -- (0,0);
            \draw (0,0) -- (1,0) -- (.5,0.86602540378) -- (0,0);  
            \end{scope}
            \begin{scope}[xshift = 1.86602540378 cm, yshift = -.5cm]
                \draw (0,0) -- (0,1) -- (-0.86602540378, .5) -- (0,0);
                \draw (0,0) -- (0,1) -- (0.86602540378,.5) -- (0,0);
            \end{scope}
        \end{scope}
        }
    }
\end{scope}
\end{scope}

\begin{scope}[xshift = 3cm, yshift = -7cm]
\draw (1.25,-.5) node {$(3.4.6.4)$};
 \clip (0,0) rectangle (2.5,2.5);
\begin{scope}[scale = .28, xshift = .25cm, yshift = .25cm]
 \foreach \i in {0,...,3}
    {
    \foreach \j in {0,...,5}
        {
        \pgfmathsetmacro{\ii}{\i * 2 * (1 + 0.86602540378 + 0.5)}
        \pgfmathsetmacro{\jj}{\j * 2 * (1 + 0.86602540378 - 0.5)}
        \begin{scope}[xshift = \ii cm, yshift = \jj cm]
            \draw (0,0) -- (0,1) -- (1,1) -- (1,0) -- (0,0);
            \draw (0,0) -- (-0.86602540378, .5) -- (0,1);
            \draw (1,0) -- (1.86602540378, .5) -- (1,1);
            \draw (0,0) -- (-.5,-0.86602540378);
            \draw (0,1) -- (-.5,1.86602540378);
            \draw (1,0) -- (1.5,-0.86602540378);
            \draw (1,1) -- (1.5,1.86602540378);
            \begin{scope}[xshift = -0.86602540378 cm, yshift = .5 cm]
                \draw (0,0) -- (-.5,-0.86602540378);
                \draw (0,0) -- (-.5,0.86602540378);
            \end{scope}
            \begin{scope}[xshift = 1.86602540378 cm, yshift = .5 cm]
                \draw (0,0) -- (.5,-0.86602540378);
                \draw (0,0) -- (.5,0.86602540378);
            \end{scope}
             \pgfmathsetmacro{\lx}{1 + 0.86602540378 + 0.5}
             \pgfmathsetmacro{\ly}{1 + 0.86602540378 - 0.5}
                \begin{scope}[xshift = \lx cm, yshift = \ly cm]
                    \draw (0,0) -- (0,1) -- (1,1) -- (1,0) -- (0,0);
                    \draw (0,0) -- (-0.86602540378, .5) -- (0,1);
                    \draw (1,0) -- (1.86602540378, .5) -- (1,1);
                    \draw (0,0) -- (-.5,-0.86602540378);
                    \draw (0,1) -- (-.5,1.86602540378);
                    \draw (1,0) -- (1.5,-0.86602540378);
                    \draw (1,1) -- (1.5,1.86602540378);
                    \begin{scope}[xshift = -0.86602540378 cm, yshift = .5 cm]
                        \draw (0,0) -- (-.5,-0.86602540378);
                        \draw (0,0) -- (-.5,0.86602540378);
                    \end{scope}
                    \begin{scope}[xshift = 1.86602540378 cm, yshift = .5 cm]
                        \draw (0,0) -- (.5,-0.86602540378);
                        \draw (0,0) -- (.5,0.86602540378);
                    \end{scope}
            \end{scope}
            
        \end{scope}
        }
    }
\end{scope}
\end{scope}

\begin{scope}[xshift = 6cm, yshift = -7cm]
\draw (1.25,-.5) node {$(4.6.12)$};
 \clip (0,0) rectangle (2.5,2.5);
\begin{scope}[scale = .2, xshift = .25cm, yshift = .25cm]
  \foreach \i in {0,...,2}
    {
    \foreach \j in {0,...,3}
        {
            \pgfmathsetmacro{\ii}{\i * 2*(1.5 + 3* 0.86602540378)} 
            \pgfmathsetmacro{\jj}{\j* 2*(1.5 + 0.86602540378)}

        \begin{scope}[xshift = \ii cm, yshift = \jj cm]
            \draw (0,0) -- (0,1) -- (1,1) -- (1,0) -- (0,0);
            \begin{scope}[xshift = -0.86602540378cm, yshift = .5cm]
                \draw (30:1) -- (90:1) -- (150:1) -- (210:1) -- (270:1) -- (330:1) -- (30:1);
            \end{scope}
            \draw (-0.86602540378, 1.5) -- (-0.86602540378-.5, 1.5+0.86602540378); 
            \draw (-2* 0.86602540378, 1) -- (-2*0.86602540378-.5,1.86602540378); 
            \draw (-0.86602540378, -.5) -- (-0.86602540378-.5, -.5-0.86602540378); 
            \draw (-2* 0.86602540378, 0) -- (-2*0.86602540378-.5,-.86602540378); 
            \begin{scope}[xshift = 1.86602540378cm, yshift = .5cm]
                \draw (30:1) -- (90:1) -- (150:1) -- (210:1) -- (270:1) -- (330:1) -- (30:1);
            \end{scope}
                \draw (1.86602540378, 1.5) -- (1.86602540378+.5, 1.5+0.86602540378); 
                \draw (1+2* 0.86602540378, 1) -- (1+2*0.86602540378+.5,1.86602540378); 
                \draw (1.86602540378, -.5) -- (1.86602540378+.5, -.5-0.86602540378); 
                \draw (1+2* 0.86602540378, 0) -- (1+2*0.86602540378+.5,-.86602540378); 
            \pgfmathsetmacro{\lx}{1.5 + 3* 0.86602540378} 
            \pgfmathsetmacro{\ly}{1.5 + 0.86602540378} 
            
            \begin{scope}[xshift = \lx cm, yshift = \ly cm]

            \draw (0,0) -- (0,1) -- (1,1) -- (1,0) -- (0,0);
            \begin{scope}[xshift = -0.86602540378cm, yshift = .5cm]
                \draw (30:1) -- (90:1) -- (150:1) -- (210:1) -- (270:1) -- (330:1) -- (30:1);
            \end{scope}
            \draw (-0.86602540378, 1.5) -- (-0.86602540378-.5, 1.5+0.86602540378); 
            \draw (-2* 0.86602540378, 1) -- (-2*0.86602540378-.5,1.86602540378); 
            \draw (-0.86602540378, -.5) -- (-0.86602540378-.5, -.5-0.86602540378); 
            \draw (-2* 0.86602540378, 0) -- (-2*0.86602540378-.5,-.86602540378); 
            \begin{scope}[xshift = 1.86602540378cm, yshift = .5cm]
                \draw (30:1) -- (90:1) -- (150:1) -- (210:1) -- (270:1) -- (330:1) -- (30:1);
            \end{scope}
                \draw (1.86602540378, 1.5) -- (1.86602540378+.5, 1.5+0.86602540378); 
                \draw (1+2* 0.86602540378, 1) -- (1+2*0.86602540378+.5,1.86602540378); 
                \draw (1.86602540378, -.5) -- (1.86602540378+.5, -.5-0.86602540378); 
                \draw (1+2* 0.86602540378, 0) -- (1+2*0.86602540378+.5,-.86602540378); 

            \end{scope}

        \end{scope}
        }
    }
\end{scope}
\end{scope}

\begin{scope}[xshift = 9cm, yshift = -7cm]
\draw (1.25,-.5) node {$(3^4.6)$};
 \clip (0,0) rectangle (2.5,2.5);
\begin{scope}[scale = .3, xshift = -0.05cm, yshift = -0.05cm]
  \foreach \i in {0,...,5}
    {
    \foreach \j in {-1,...,4}
        {
        \pgfmathsetmacro{\ii}{\i * 2.5 + \j * .5} 
        \pgfmathsetmacro{\jj}{\i * 0.86602540378 + \j * 3 * 0.86602540378} 
        \begin{scope}[xshift = \ii cm, yshift = \jj cm]
        \draw (0:1) -- (60:1) -- (120:1) -- (180:1) -- (240:1) -- (300:1) -- (0:1);
         \pgfmathsetmacro{\r}{2* 0.86602540378} 
        \draw   (0:1)   -- (30:\r) -- 
                (60:1)  -- (90:\r) --
                (120:1) -- (150:\r) --
                (180:1) -- (210:\r) --
                (240:1) -- (270:\r) --
                (300:1) -- (330:\r) -- (0:1);
        \draw  (0:2) -- (60:2) -- (120:2) -- (180:2) -- (240:2) -- (300:2) -- (0:2);
        \draw  (0:1)   -- (0:2);
        \draw  (60:1)  -- (60:2);
        \draw  (120:1) -- (120:2);
        \draw  (180:1) -- (180:2);
        \draw  (240:1) -- (240:2);
        \draw  (300:1) -- (300:2);
        \end{scope}
        }
    }
\end{scope}
\end{scope}

 \end{tikzpicture}

\caption{The $11$ Archimedean tilings. The three Platonic lattices are represented on the first line.
The first two of the second line correspond to the Kagome and Super-Kagome lattices. Figure reproduced from \cite{PeTa21}.}
\label{fig:11_tilings}
\end{figure}
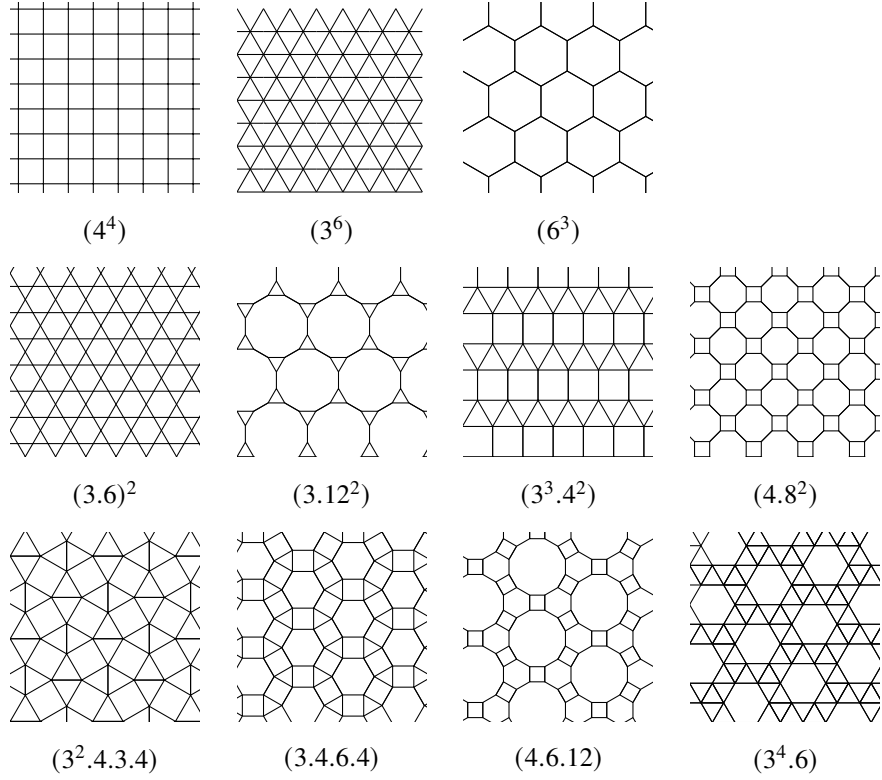

\subsection{Flat bands for Kagome and Super-Kagome tilings}\label{subsec:kagomeandsuper}

We start by studying the Kagome lattice $(3.6)^2$. 
We know from \cite{PeTa21} that the operator $-\Delta_\V\big((3.6)^2\big)$ has a flat band at $6$. 
In fact, in \cite{PeTa21} the authors study the normalised Laplacian that corresponds to our Laplacian 
\eqref{eq:Deltavertices} divided by $\deg_0(\cdot)$. Since $\deg_0(\cdot)$ is constant for the Archimedean lattices, 
we can obtain the band functions of $-\Delta_\V$ by just multiplying their band functions by $\deg_0(\XX)$. 
Hence, it only remains to study $-\Delta_\F\big((3.6)^2\big)$, which can be obtained by studying
$-\Delta_\F(\xi)=\dd_1(\xi)\dd_1(\xi)^*$ for all $\xi \in \T^2$.
Since $\dd_1(\xi)$ was already computed in \eqref{eq:d1Kagome} we can directly get that $-\Delta_\F(\xi)$ 
is given by
\begin{equation*}
\scalemath{0.90}{\begin{pmatrix}
6 
  & -1 - e^{-2\pi i\xi_1} - e^{2\pi i(\xi_2-\xi_1)}
  & -e^{-2\pi i\xi_1} - e^{2\pi i(\xi_2-2\xi_1)} - e^{2\pi i(\xi_2-\xi_1)}) \\[4pt]
-1 - e^{2\pi i\xi_1} - e^{-2\pi i(\xi_2-\xi_1)}) 
  & 3 
  & 0 \\[4pt]
-e^{2\pi i\xi_1} - e^{-2\pi i(\xi_2-2\xi_1)} - e^{-2\pi i(\xi_2-\xi_1)}
  & 0 
  & 3
\end{pmatrix}}.
\end{equation*}
The corresponding characteristic polynomial, defined by
$$
p(\lambda; \xi_1,\xi_2) :=\det\Big(-\Delta_\F(\xi_1,\xi_2)-\lambda I_\nf\Big),
$$
reads
\begin{align*}
p(\lambda; \xi_1,\xi_2) 
=
(\lambda-3)\Big[-\lambda^2+9\lambda-12
+ 4\big(\cos(2 \pi \xi_1) + \cos(2 \pi (\xi_1 - \xi_2)) +  \cos(2 \pi \xi_2)\big)\Big].
\end{align*}
As a consequence, the eigenvalues of $-\Delta_\F(\xi_1,\xi_2)$ are
given by
\begin{align*}
\lambda_1(\xi_1,\xi_2) & = 3, \\
\lambda_2(\xi_1,\xi_2) &  = \frac{9 - \sqrt{33 + 16\big(\cos(2 \pi \xi_1) + \cos(2 \pi (\xi_1 - \xi_2)) +  \cos(2 \pi \xi_2)\big)}}{2}, \\
\lambda_3(\xi_1,\xi_2) & = \frac{9 + \sqrt{33 + 16\big(\cos(2 \pi \xi_1) + \cos(2 \pi (\xi_1 - \xi_2)) +  \cos(2 \pi \xi_2)\big)}}{2}.
\end{align*}
We then conclude from Theorem \ref{theo:Dm} that $D_m\big((3.6)^2)\big)$ has four flat bands: 
two related to $-\Delta_\V$ at $\pm\sqrt{6+m^2}$ and two related to $-\Delta_\F$ at $\pm\sqrt{3+m^2}$. 

\begin{Remark}\label{rem_invariance}
1) It is readily observed that the expressions for the matrices $\dd_1(\xi)$ and
$\dd_1(\xi)\dd_1(\xi)^*$ depend on the choice of $\V_0$, $\A_0$, and $\F_0$.
However, the expression for the above characteristic polynomial and for the eigenvalues does not depend on this choice.
On the other hand, these expressions also depend on the choice of generators for the action
of $\Z^2$, but the spectrum of $-\Delta_\F$ (obtained by taking a union over all $\xi \in \T^2$)
does not. We refer to \cite{KR} for a precise discussion about these dependences.

2) We also note that the function
\begin{equation*}
\T^2\ni (\xi_1,\xi_2) \mapsto \cos(2 \pi \xi_1) + \cos(2 \pi (\xi_1 - \xi_2)) +  \cos(2 \pi \xi_2)\in \C
\end{equation*} 
has been studied in \cite[Sec.~3.1]{PeTa21}, and similar investigations would directly lead 
to a formula for the integrated density of states. 
\end{Remark}

\begin{figure}[t]
\centering
\begin{minipage}[t]{0.58\textwidth}
  \centering
  \begin{tikzpicture}[scale=2]
  \begin{scope}[yshift=-3.5cm]
  \clip (0,0) rectangle (2.5,2);
  \coordinate (P) at (0.425,0.725);
  \pgfmathsetmacro{\sq}{0.86602540378}
  \pgfmathsetmacro{\aonex}{2*\sq+2}
  \pgfmathsetmacro{\aoney}{0}
  \pgfmathsetmacro{\atwox}{(\sq+1)*0.225}
  \pgfmathsetmacro{\atwoy}{(2*\sq+1.5)*0.225}
  \coordinate (aone) at ({\aonex*0.225},{\aoney*0.225});
  \coordinate (atwo) at ({\atwox},{\atwoy});
  \coordinate (P1)  at ($(P)+(aone)$);
  \coordinate (P2)  at ($(P)+(atwo)$);
  \coordinate (P12) at ($(P)+(aone)+(atwo)$);
  \fill[blue!10, opacity=0.45] (P) -- (P1) -- (P12) -- (P2) -- cycle;
  \draw[->, thick, blue!80] (P) -- (P1) node[midway, below] {$a_1$};
  \draw[->, thick, blue!80] (P) -- (P2) node[midway, left]  {$a_2$};
  \begin{scope}[scale=.225]
    \foreach \i in {0,...,6}{
      \foreach \j in {0,...,3}{
        \pgfmathsetmacro{\ii}{\i*(2*0.86602540378+2)}
        \pgfmathsetmacro{\jj}{\j*(4*0.86602540378+3)}
        \begin{scope}[xshift=\ii cm, yshift=\jj cm]
          \draw (0,0)--(1,0)--(.5,0.86602540378)--(0,0);
          \draw (.5,0.86602540378)--(.5,1.86602540378);
          \draw (.5,1.86602540378)--(1,1.86602540378+0.86602540378)
                --(0,1.86602540378+0.86602540378)--(.5,1.86602540378);
          \draw (210:1)--(0,0);
          \begin{scope}[xshift=1cm]
            \draw (330:1)--(0,0);
          \end{scope}
          \pgfmathsetmacro{\l}{2*0.86602540378+1}
          \begin{scope}[yshift=\l cm]
            \draw (150:1)--(0,0);
          \end{scope}
          \begin{scope}[xshift=1cm,yshift=\l cm]
            \draw (30:1)--(0,0);
          \end{scope}
        \end{scope}
        \pgfmathsetmacro{\iii}{\ii+(0.86602540378+1)}
        \pgfmathsetmacro{\jjj}{\jj+(2*0.86602540378+1.5)}
        \begin{scope}[xshift=\iii cm,yshift=\jjj cm]
          \draw (0,0)--(1,0)--(.5,0.86602540378)--(0,0);
          \draw (.5,0.86602540378)--(.5,1.86602540378);
          \draw (.5,1.86602540378)--(1,1.86602540378+0.86602540378)
                --(0,1.86602540378+0.86602540378)--(.5,1.86602540378);
          \draw (210:1)--(0,0);
          \begin{scope}[xshift=1cm]
            \draw (330:1)--(0,0);
          \end{scope}
          \pgfmathsetmacro{\l}{2*0.86602540378+1}
          \begin{scope}[yshift=\l cm]
            \draw (150:1)--(0,0);
          \end{scope}
          \begin{scope}[xshift=1cm,yshift=\l cm]
            \draw (30:1)--(0,0);
          \end{scope}
        \end{scope}
      }
    }
  \end{scope}
  \end{scope}
  
  \end{tikzpicture}
  
  \vspace{2pt}
  
  {\footnotesize (a) $(3.12^2)$ Super-Kagome lattice and a fundamental domain.}
  
\end{minipage}\hfill
\begin{minipage}[t]{0.38\textwidth}
  \centering
\begin{tikzpicture}[scale=3,
  vtx/.style={draw, circle, fill=white, inner sep=1pt},
  edge/.style={line width=0.55pt}
]
\draw[edge, gray!50] (0,0) rectangle (1,1);

\coordinate (v1) at (0,     0);
\coordinate (v2) at (0.134, 0);
\coordinate (v3) at (0,     0.268);
\coordinate (v4) at (0.845, 0.577);
\coordinate (v5) at (0.577, 0.845);
\coordinate (v6) at (0.845, 0.845);

\coordinate (v1R)  at (1,     0);
\coordinate (v3R)  at (1,     0.268);
\coordinate (v1T)  at (0,     1);
\coordinate (v2T)  at (0.134, 1);
\coordinate (v1TR) at (1,     1);

\fill[gray!15] (v1)--(v2)--(v3)--cycle;
\fill[gray!30] (v4)--(v5)--(v6)--cycle;

\draw[edge] (v1)--(v2)--(v3)--(v1);
\draw[edge] (v4)--(v5)--(v6)--(v4);

\draw[edge] (v5)--(v2T);
\draw[edge] (v6)--(v1TR);
\draw[edge] (v4)--(v3R);

\node[vtx, label={[font=\tiny]below left:$\xx_1$}] at (v1)  {};
\node[vtx, label={[font=\tiny]below:$\xx_2$}]      at (v2)  {};
\node[vtx, label={[font=\tiny]left:$\xx_3$}]       at (v3)  {};
\node[vtx, label={[font=\tiny]left:$\xx_4$}]       at (v4)  {};
\node[vtx, label={[font=\tiny]above:$\xx_5$}]      at (v5)  {};
\node[vtx, label={[font=\tiny]above:$\xx_6$}]      at (v6)  {};

\node[vtx] at (v1R)  {};
\node[vtx] at (v3R)  {};
\node[vtx] at (v1T)  {};
\node[vtx] at (v2T)  {};
\node[vtx] at (v1TR) {};

\end{tikzpicture}
  
  \vspace{2pt}
  
  {\footnotesize (b) Quotient graph on the flat torus.}
\end{minipage}
\caption{(a) $(3.12^2)$ Super-Kagome lattice with a fundamental domain
for the $\Z^2$-action. (b) The induced quotient graph embedded in the
flat torus model of $\T^2$. In this case $\nv=6$, $\ne=9$, and $\nf=3$.}
\label{fig:superkagome_and_quotient}
\end{figure}
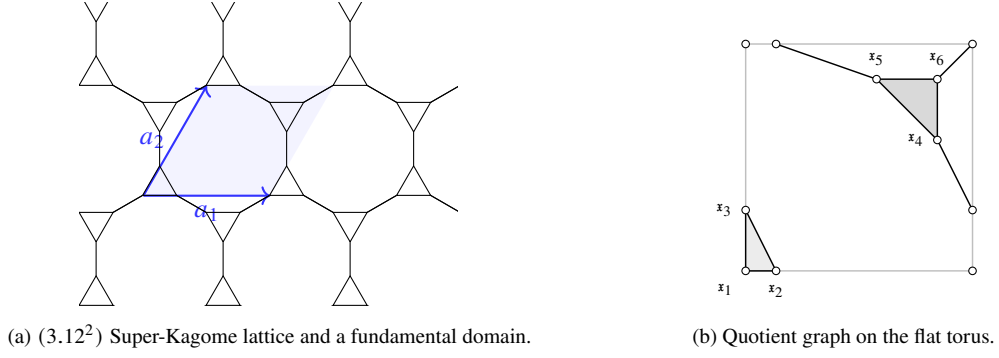

Let us turn our attention to the $(3.12^2)$ tiling also dubbed \emph{Super-Kagome lattice} in \cite{KeTaWi24}. 
We already know from \cite{PeTa21} that $-\Delta_\V$ has two flat bands, at $3$ and at $5$.  
In addition, for a suitable choice of $\V_0$, $\A_0$, and $\F_0$ one gets 
\begin{equation*}
\dd_1(\xi_1,\xi_2)=
\scalemath{0.80}{
\begin{pmatrix}
e^{2\pi i\xi_2}& 1 & e^{2\pi i\xi_1} & 1-e^{2\pi i\xi_1} & 1 & e^{2\pi i(\xi_1-\xi_2)} & e^{2\pi i\xi_1} & 1-e^{2\pi i(\xi_1-\xi_2)} & e^{2\pi i(\xi_1-\xi_2)}-e^{2\pi i\xi_1} \\ 
1& 1 & 1 & 0 & 0 & 0 & 0 & 0 & 0 \\ 
0& 0 & 0 & 0 & 1 & 1 & 1 & 0 & 0 
\end{pmatrix}.}
\end{equation*}
By a straightforward computation, one gets an expression for the square matrix 
$\dd_1(\xi)\dd_1(\xi)^*$, namely
\begin{equation}\label{eq:facesuper}
-\Delta_\F(\xi)=
\begin{pmatrix}
12 - 2F(\xi_1,\xi_2)
  & 1+e^{2\pi i\xi_1}+e^{2\pi i\xi_2} 
  & 1+e^{2\pi i\xi_1}+e^{2\pi i(\xi_1-\xi_2)} \\
1+e^{-2\pi i\xi_1}+e^{-2\pi i\xi_2} & 3 & 0 \\
1+e^{-2\pi i\xi_1}+e^{-2\pi i(\xi_1-\xi_2)} & 0 & 3
\end{pmatrix},
\end{equation}
with 
$F(\xi_1,\xi_2):=\cos(2\pi \xi_1) + \cos(2\pi (\xi_1 - \xi_2)) + \cos(2\pi \xi_2)$.
The characteristic polynomial then reads
\begin{equation}\label{eq:superpoly}
p(\lambda; \xi_1,\xi_2) =
(\lambda-3)\Big[-\lambda^2 +15\lambda -30 -2(\lambda-5)F(\xi_1,\xi_2)\Big].
\end{equation}
Clearly $\lambda=3$ is a common eigenvalue of $-\Delta_\F(\xi)$ for all $(\xi_1,\xi_2)\in \T^2$,
and this turns out to be the only common eigenvalue. Indeed, 
since the constant function $1$ and $F$ are linearly independent, one observes that the second term in \eqref{eq:superpoly} would vanish identically in $\T^2$ if and only if $-\lambda^2 +15\lambda-30$ and $\lambda-5$ share a common root. 
Since this is not the case, we conclude that $D_m\big((3.12^2)\big)$ has six flat bands: 
two related to $-\Delta_\V$ at $\pm\sqrt{5+m^2}$ and four related to both $-\Delta_\V$ and $-\Delta_\F$ 
at $\pm\sqrt{3+m^2}$. In this case, the new flat bands coincide with the previous ones.

\begin{Remark}
One of the interesting results obtained in \cite{KeTaWi24} is the fact that if we consider 
a periodically weighted version of the Laplacian on $(3.12^2)$, it is possible to have either $0$, $1$, or $2$ flat bands. 
However, the $1$ and $2$ flat-bands are disconnected in the space of parameters indexing such perturbations. 
It would be interesting to understand what happens now since we have three sources of flat bands at our disposal,  
and in particular since some flat bands coincide in value but have support in orthogonal spaces.
\end{Remark}

\subsection{Absence of flat bands for all other Archimedean tilings}

In this section, we obtain that for the remaining nine Archimedean lattices, 
the operator $D_m$ does not exhibit any flat band.
The argument is based on Theorem \ref{theo:Dm} and once again we can reduce our investigation
to the study of $-\Delta_\F$ since it is shown in \cite{KeTaWi24}
that the operators $-\Delta_\V$ do not possess flat bands. 

The argument is similar to the one already used for the study of the Super-Kagome lattice:
we exhibit the matrices $\dd_1(\xi_1,\xi_2)$ (which depends on the choice of the basis 
for $\V_0$, $\A_0$, and $\F_0$, as emphasized in Remark \ref{rem_invariance}), 
and then provide the characteristic polynomials for $-\Delta_\F(\xi)=\dd_1(\xi)\dd_1(\xi)^*$.
Since the computation of the matrices $\dd_1(\xi)\dd_1(\xi)^*$ is straightforward, we do not 
present their expressions.
The characteristic polynomials are presented in such a way that 
the coefficients of all linearly independent functions are clearly visible.
It remains then to observe that for none of them there exists a solution $\lambda$ of the equation
$p(\lambda; \xi_1,\xi_2)=0$ which is independent of $(\xi_1,\xi_2)$. Indeed, such a solution
would cancel all mentioned coefficients, which is clearly impossible in all the examples
presented below. In fact, it turns out that there always exists at least one trigonometric function
of $(\xi_1,\xi_2)$ whose coefficient does not depend on $\lambda$. 
This automatically leads to the absence of a solution of the characteristic equation
independent of $\xi$.

In the rest of this section, we provide the matrices $\dd_1(\xi_1,\xi_2)$
and the characteristic polynomials $p(\lambda; \xi_1,\xi_2)$ for the matrices 
$-\Delta_\F(\xi_1,\xi_2)$ of the remaining nine Archimedean lattices.

\paragraph{$(4^4)$ Square lattice}

\begin{equation*}
\dd_1(\xi_1,\xi_2)=\begin{pmatrix}
1 -e^{2 \pi i \xi_2} & e^{2 \pi i \xi_1}-1
\end{pmatrix}.
\end{equation*}

\begin{equation*}
p(\lambda; \xi_1,\xi_2) =-\lambda+4-2\cos(2\pi\xi_1)-2\cos(2\pi\xi_2). 
\end{equation*}

\paragraph{$(3^6)$ Triangular lattice}

\begin{equation*}
\dd_1(\xi_1,\xi_2)=\begin{pmatrix}
1& -1 & e^{2 \pi i \xi_1}\\
-e^{-2\pi i(\xi_1-\xi_2)}&1 &-1
\end{pmatrix}.
\end{equation*}

\begin{align*}
p(\lambda; \xi_1,\xi_2)
=6 - 6\lambda + \lambda^2
- 2\cos(2 \pi \xi_1) - 2\cos(2 \pi (\xi_1 - \xi_2)) - 2\cos(2 \pi \xi_2). 
\end{align*}

\paragraph{$(6^3)$ Hexagonal lattice}

\begin{equation*}
\dd_1(\xi_1,\xi_2)=\begin{pmatrix}
1 -e^{-2 \pi i (\xi_1-\xi_2)} & -1+e^{2 \pi i \xi_2} &-e^{2 \pi i \xi_2}+e^{-2 \pi i (\xi_1-\xi_2)}
\end{pmatrix}.
\end{equation*}

\begin{equation*}
p(\lambda; \xi_1,\xi_2) =-\lambda+6-2\cos(2\pi\xi_1) -2\cos(2\pi(\xi_1-\xi_2)) -2\cos(2\pi\xi_2).
\end{equation*}

\paragraph{$(3^3.4^2)$ Elongated lattice}

\begin{equation*}
\dd_1(\xi_1,\xi_2)=
\begin{pmatrix}
1-e^{2i\pi\xi_1} & 1 & -1 & 0 &0 \\
0 & 0 & 1 & -e^{2i\pi(\xi_1-\xi_2)} & e^{2i\pi(\xi_1-\xi_2)}\\
0 & -1 & 0 & 1 & -e^{2i\pi\xi_1}
\end{pmatrix}.
\end{equation*}

\begin{align*}
p(\lambda; \xi_1,\xi_2)
&=24-29\lambda+10\lambda^2-\lambda^3 +2(-11+7\lambda-\lambda^2) \cos(2\pi\xi_1) \\
&\quad +2\cos(4\pi\xi_1)-2\cos(2\pi(\xi_1-\xi_2))-2\cos(2\pi\xi_2).
\end{align*}

\paragraph{$(4.8^2)$ Truncated square lattice}

\begin{equation*}
\dd_1(\xi_1,\xi_2)=\begin{pmatrix}
1&1&1&1&0&0 \\
1& e^{-2i\pi\xi_1}&e^{-2i\pi(\xi_1+\xi_2)}&e^{-2i\pi\xi_2}&e^{-2i\pi\xi_2}-1&1-e^{-2i\pi\xi_1}
\end{pmatrix}.
\end{equation*}

\begin{align*}
p(\lambda; \xi_1,\xi_2)
&=28-12\lambda+\lambda^2 +2(-6+\lambda)\big(\cos(2\pi\xi_1)+\cos(2\pi\xi_2)\big)\\
&\quad -2\cos(2\pi(\xi_1-\xi_2))-2\cos(2\pi(\xi_1+\xi_2)).
\end{align*}

\paragraph{$(3^2.4.3.4)$ Snub square lattice}

\begin{align*}
&\dd_1(\xi_1,\xi_2) \\
& =\begin{pmatrix}
1& 1 & 1 & 1 & 0 & 0 & 0 & 0 & 0 & 0 \\ 
0& -1 & 0 & 0 & 0 & 0 & -1 & 0 & 0 & 1 \\
0& 0 & 0 & -e^{-2\pi i\xi_1} & -e^{-2\pi i\xi_1} & 0 & 0 & 0 & 0 & -1 \\ 
0& 0 & 0 & 0 & e^{-2\pi i\xi_1} & 1 & e^{2\pi i(\xi_2-\xi_1)} & e^{2\pi i(\xi_2-2\xi_1)} & 0 & 0 \\ 
-1& 0 & 0 & 0 & 0 & -1 & 0 & 0 & 1 & 0 \\ 
0& 0 & e^{2\pi i(\xi_2-\xi_1)} & 0 & 0 & 0 & 0 & e^{2\pi i(\xi_2-\xi_1)} & 1 & 0  
\end{pmatrix}\end{align*}

\begin{align*}
p(\lambda; \xi_1,\xi_2)
&= 340- 1088 \lambda + 1180 \lambda^2 - 600 \lambda^3 + 156 \lambda^4 - 20 \lambda^5 + \lambda^6 \\
&\quad + 2\cos(4 \pi \xi_1) + 2\cos(4 \pi (\xi_1 - \xi_2)) \\
&\quad +  2(-14 + 8 \lambda - \lambda^2)\big(\cos(2 \pi (2 \xi_1 - \xi_2)) + \cos(2 \pi \xi_2)\big) \\
&\quad +  4(-36 + 36 \lambda - 11 \lambda^2 + \lambda^3)\big(\cos(2 \pi \xi_1)+ \cos(2 \pi (\xi_1 - \xi_2))\big).
\end{align*}

\paragraph{$(3.4.6.4)$ Small rhombi-trihexagonal lattice}

\[
\dd_1(\xi_1,\xi_2) 
=\scalemath{0.90}{\begin{pmatrix}
1 & 1 & 1 & 1 & 1 & 1 & 0 & 0 & 0 & 0 & 0 & 0 \\ 
1 & 0 & 0 & e^{-2\pi i\xi_2} & 0 & 0 & 0 & 0 & e^{-2\pi i\xi_2} & 1 & 0 & 0 \\ 
0 & 1 & 0 & 0 & e^{2\pi i(\xi_1-\xi_2)} & 0 & 0 & 1 & 0 & 0 & 0 & e^{2\pi i(\xi_1-\xi_2)} \\ 
0 & 0 & 1 & 0 & 0 & e^{2\pi i\xi_1} & e^{2\pi i\xi_1} & 0 & 0 & 0 & 1 & 0 \\ 
0 & 0 & 0 & 0 & 0 & 0 & 0 & 0 & 0 & 1 & e^{-2\pi i\xi_2} & e^{2\pi i(\xi_1-\xi_2)} \\ 
0 & 0 & 0 & 0 & 0 & 0 & e^{2\pi i\xi_1} & 1 & e^{2\pi i(\xi_1-\xi_2)} & 0 & 0 & 0
\end{pmatrix}}
\]

\begin{align*}
&p(\lambda; \xi_1,\xi_2)\\
&=1146- 2880 \lambda + 2520 \lambda^2 - 1048 \lambda^3 + 225 \lambda^4 - 24 \lambda^5 + \lambda^6\\
&\quad + 2\cos(4  \pi \xi_1) - 4\cos(2 \pi (\xi_1 - 2 \xi_2)) +2 \cos( 4 \pi (\xi_1 - \xi_2) ) \\
&\quad - 4\cos(2 \pi (2 \xi_1 - \xi_2))+ 2\cos(4 \pi \xi_2)- 4 \cos(2 \pi (\xi_1 + \xi_2))  \\
&\quad +  2(-190 + 192 \lambda - 76 \lambda^2 + 14 \lambda^3 - \lambda^4) 
\big(\cos(2 \pi \xi_1)+ \cos(2 \pi (\xi_1 - \xi_2)) + \cos(2 \pi \xi_2)\big).
\end{align*}

\paragraph{$(4.6.12)$ Great rhombi-trihexagonal lattice}

\begin{align*}
&\dd_1(\xi_1,\xi_2) =\\
& \scalemath{0.68}{\begin{pmatrix}
0 & e^{2\pi i(\xi_2-\xi_1)} & 0 & e^{-2\pi i\xi_1} & e^{-2\pi i\xi_1} & e^{-2\pi i\xi_1} & 0 & e^{-2\pi i\xi_2} & 0 & 1 & 1 & 1 & 0 & 0 & -e^{-2\pi i\xi_2} & e^{-2\pi i\xi_2} & e^{-2\pi i\xi_1} & -1 \\
1 & 0 & 0 & 0 & 0 & 0 & e^{-2\pi i\xi_2} & 0 & 0 & 0 & 0 & 0 & 0 & 0 & e^{-2\pi i\xi_2} & 0 & -e^{-2\pi i\xi_2} & 0 \\ 
1 & 1 & 1 & 1 & 0 & 0 & 0 & 0 & 0 & 0 & 0 & 1 & -1 & 0 & 0 & 0 & 0 & 0 \\ 
0 & 0 & 0 & 0 & 1 & 0 & 0 & 0 & 0 & 0 & 1 & 0 & 1 & 1 & 0 & 0 & 0 & 0 \\ 
0 & 0 & 0 & 0 & 0 & 1 & 1 & 1 & 1 & 1 & 0 & 0 & 0 & -1 & 0 & 0 & 0 & 0 \\ 
0 & 0 & e^{2\pi i(\xi_2-\xi_1)} & 0 & 0 & 0 & 0 & 0 & 1 & 0 & 0 & 0 & 0 & 0 & 0 & -1 & 0 & 1 
\end{pmatrix}}
\end{align*}

\begin{align*}
&p(\lambda; \xi_1,\xi_2)\\
&= 13434- 20832 \lambda + 12048 \lambda^2 - 3400 \lambda^3 + 498 \lambda^4 - 36 \lambda^5 + \lambda^6\\
&\quad+ 2\cos(4 \pi \xi_1) + 2\cos(4 \pi (\xi_1 - \xi_2)) + 2\cos(4 \pi \xi_2) \\
&\quad +  2(-34 + 12 \lambda - \lambda^2)\big(\cos(2 \pi (\xi_1 - 2 \xi_2)) +\cos(2 \pi (2 \xi_1 - \xi_2)) +\cos(2 \pi (\xi_1 + \xi_2))\big) \\
&\quad + 2(-2206 + 1732 \lambda - 503 \lambda^2 + 64 \lambda^3 - 3 \lambda^4)\big(\cos(2 \pi \xi_1) + \cos(2 \pi (\xi_1 - \xi_2)) + \cos(2 \pi \xi_2)\big).
\end{align*}

\paragraph{$(6.3^4)$ Snub hexagonal lattice}

\begin{align*}
&\dd_1(\xi_1,\xi_2) =\\
&\scalemath{0.90}{\begin{pmatrix}
1 & 1 & 1 & 1 & 1 & 1 & 0 & 0 & 0 & 0 & 0 & 0 & 0 & 0 & 0 \\ 
0 & 0 & 0 & 1 & 0 & 0 & e^{2\pi i\xi_2} & 0 & 0 & 0 & 0 & 1 & 0 & 0 & 0 \\
0 & 0 & 1 & 0 & 0 & 0 & 0 & 0 & -1 & 0 & 0 & 0 & 0 & 1 & 0 \\ 
0 & 1 & 0 & 0 & 0 & 0 & 0 & -1 & 0 & 0 & 1 & 0 & 0 & 0 & 0 \\
1 & 0 & 0 & 0 & 0 & 0 & -1 & 0 & 0 & 0 & 0 & 0 & 1 & 0 & 0 \\ 
0 & 0 & 0 & 0 & 0 & 1 & 0 & 0 & e^{-2\pi i\xi_1} & 1 & 0 & 0 & 0 & 0 & 0 \\
0 & 0 & 0 & 0 & 1 & 0 & 0 & e^{2\pi i(\xi_2-\xi_1)} & 0 & 0 & 0 & 0 & 0 & 0 & 1 \\ 
0 & 0 & 0 & 0 & 0 & 0 & 0 & 0 & 0 & 0 & 0 & 0 & e^{2\pi i(\xi_2-\xi_1)} & e^{-2\pi i\xi_1} & 1 \\
0 & 0 & 0 & 0 & 0 & 0 & 0 & 0 & 0 & 1 & e^{-2\pi i\xi_1} & e^{-2\pi i\xi_2} & 0 & 0 & 0 
\end{pmatrix}}
\end{align*}

\begin{align*}
&p(\lambda; \xi_1,\xi_2)\\
&=5034 - 26112 \lambda + 50784 \lambda^2 - 51518 \lambda^3 + 30858 \lambda^4 - 11487 \lambda^5 + 2682 \lambda^6 - 381 \lambda^7 + 30 \lambda^8 - \lambda^9 \\
&\quad+ 2\cos(4 \pi \xi_1)+ 2\cos(4 \pi (\xi_1 - \xi_2)) + 2\cos(4 \pi \xi_2) \\
&\quad +2(-14 + 8 \lambda - \lambda^2)\big(\cos(2 \pi (\xi_1 - 2 \xi_2)) +\cos(2 \pi (2 \xi_1 - \xi_2))
+ \cos(2 \pi (\xi_1 + \xi_2)) \big)  \\
& \quad +  2(-826 + 1752 \lambda - 1503 \lambda^2 + 667 \lambda^3 - 161 \lambda^4 + 20 \lambda^5 - \lambda^6) \\
& \qquad \times \big(\cos(2 \pi \xi_1) + \cos(2 \pi (\xi_1 - \xi_2)) + \cos(2 \pi \xi_2)\big).
\end{align*}

\subsection{The Lieb lattice}\label{subsec:lieb}
 
In this subsection, we study another periodic graph.
It has been called the \emph{Lieb lattice} after some specific investigations of E.~Lieb
in \cite{Li89}, see also \cite{JKHXLL19}.
The graph can be obtained from $\Z^2$ by adding a vertex at the midpoint of each edge, see Figure \ref{fig:lieb}.
Its interest stems from the fact that $-\Delta_\V$ exhibits a flat band; we show that,
for the Dirac--Hodge operator, this flat band is the only one, since the contribution
coming from the faces is purely dispersive.

\begin{figure}[ht]
\centering
\begin{tikzpicture}[scale=1.6,
  corner/.style={circle, fill=black, inner sep=2.2pt},
  mid/.style={circle, fill=white, draw=black, line width=0.7pt, inner sep=1.8pt},
  latedge/.style={line width=0.8pt, black!70},
  fund/.style={fill=black!8}
]

\def\imin{-1}\def\imax{2}
\def\jmin{-1}\def\jmax{2}
\pgfmathtruncatemacro{\imaxp}{\imax+1}
\pgfmathtruncatemacro{\jmaxp}{\jmax+1}

\fill[fund] (0,0) rectangle (1,1);

\foreach \i in {\imin,...,\imax}{
  \foreach \j in {\jmin,...,\jmaxp}{
    \draw[latedge] (\i,\j) -- (\i+0.5,\j) -- (\i+1,\j);
  }
}
\foreach \i in {\imin,...,\imaxp}{
  \foreach \j in {\jmin,...,\jmax}{
    \draw[latedge] (\i,\j) -- (\i,\j+0.5) -- (\i,\j+1);
  }
}

\foreach \i in {\imin,...,\imaxp}{
  \foreach \j in {\jmin,...,\jmaxp}{
    \node[corner] at (\i,\j) {};
  }
}

\foreach \i in {\imin,...,\imax}{
  \foreach \j in {\jmin,...,\jmaxp}{
    \node[mid] at (\i+0.5,\j) {};
  }
}

\foreach \i in {\imin,...,\imaxp}{
  \foreach \j in {\jmin,...,\jmax}{
    \node[mid] at (\i,\j+0.5) {};
  }
}

\draw[line width=1pt] (0,0) rectangle (1,1);
\node[corner] at (0,0) {};
\node[mid] at (0,0.5) {};
\node[mid] at (0.5,0) {};
\node at (-0.17,-0.17) {$\hx_1$};
\node at (-0.28,0.5) {$\hx_2$};
\node at (0.5,-0.28) {$\hx_3$};

\draw[line width=1.3pt] (0,0) -- (1,0);
\draw[line width=1.3pt] (0,0) -- (0,1);

\begin{scope}[xshift=3.3cm, yshift=-0.3cm]
  \draw[-{Stealth[length=2mm]}, line width=0.9pt,blue!80] (0,0) -- (0.8,0) node[right=2pt]{$a_1$}; 
  \draw[-{Stealth[length=2mm]}, line width=0.9pt,blue!80] (0,0) -- (0,0.8) node[above=2pt]{$a_2$};
\end{scope}

\end{tikzpicture}
\caption{The Lieb lattice: vertices of the underlying square lattice $\Z^2$ (black,
type $\hx_1$) together with a vertex added at the midpoint of each edge (white,
types $\hx_2$ and $\hx_3$). The shaded square is a fundamental domain for the
$\Z^2$-action generated by $a_1,a_2$.}
\label{fig:lieb}
\end{figure}
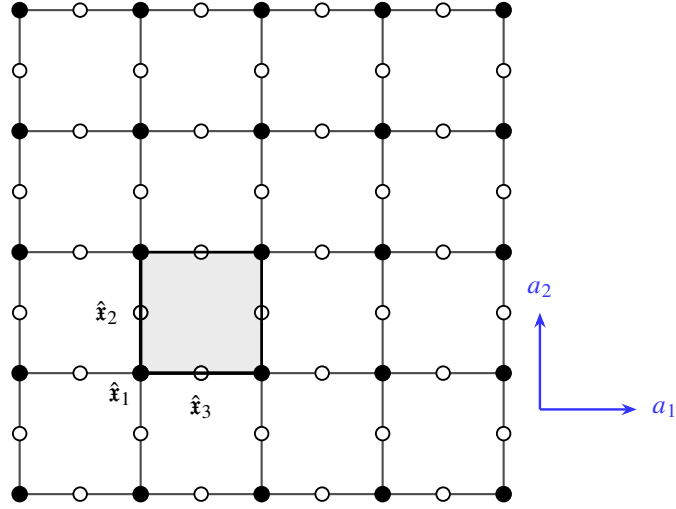

For the Lieb lattice, the quotient graph $\XX$ has $\nv=3$ vertices, $\ne=4$ unoriented
edges, and $\nf=1$ face, in accordance with relation \eqref{eq:euler0}.
We take the three vertices $\hx_1=(0,0)$, $\hx_2=(0,\tfrac12)$, $\hx_3=(\tfrac12,0)$, the four
oriented edges
$$
\he_1=(\hx_1,\hx_2),\quad \he_2=(\hx_1,\hx_3),\quad \he_3=(\hx_2,(a_2)\hx_1),\quad \he_4=(\hx_3,(a_1)\hx_1),
$$
and the single face
$$
\hf=\big(\he_2,\he_4,(a_1)\he_1,(a_1)\he_3,
(a_2)\overline{\he_4},(a_2)\overline{\he_2},\overline{\he_3},\overline{\he_1}\big),
$$
where $a_1$ is the shift by $(1,0)$ and $a_2$ is the shift by $(0,1)$.
A direct application of Propositions \ref{prop_d1xi} and \ref{prop_d2xi} then gives
\begin{equation*}
\dd_0(\xi)=
\begin{pmatrix}
-1 & 1 & 0\\
-1 & 0 & 1\\
e^{2\pi i\xi_2} & -1 & 0\\
e^{2\pi i\xi_1} & 0 & -1
\end{pmatrix},
\qquad
\dd_1(\xi)=
\begin{pmatrix}
e^{2\pi i\xi_1}-1 & 1-e^{2\pi i\xi_2} & e^{2\pi i\xi_1}-1 & 1-e^{2\pi i\xi_2}
\end{pmatrix}.
\end{equation*}
As a consequence, the graph-Laplacian on vertices $-\Delta_\V(\xi)=\dd_0(\xi)^*\dd_0(\xi)$
has characteristic polynomial
\begin{equation*}
p(\lambda; \xi_1,\xi_2)
=(2-\lambda)\Big[\lambda^2-6\lambda+4-2\cos(2\pi\xi_1)-2\cos(2\pi\xi_2)\Big],
\end{equation*}
so that its band functions are
\begin{align*}
\lambda_1(\xi_1,\xi_2)&=2,\\
\lambda_2(\xi_1,\xi_2)&=3-\sqrt{5+2\cos(2\pi\xi_1)+2\cos(2\pi\xi_2)},\\
\lambda_3(\xi_1,\xi_2)&=3+\sqrt{5+2\cos(2\pi\xi_1)+2\cos(2\pi\xi_2)}.
\end{align*}
In particular, $-\Delta_\V$ has a flat band at $2$. Note that the flat band at $2$ is not isolated: it lies at the top of the lower
dispersive band $\lambda_2(\xi_1,\xi_2)=3-\sqrt{5+2\cos(2\pi\xi_1)+2\cos(2\pi\xi_2)}$, which
attains the value $2$ only at $\xi=(\tfrac12,\tfrac12)$, where the flat band and the lower band touch.
The same graph, together with the spectral consequences of this band-touching for a
perturbed operator, was studied in \cite{MiPa25}; see also the isospectral Dirac-type
operator on $\Z^2$ treated there.
 
On the other hand, since $\nf=1$, the Laplacian on faces
$-\Delta_\F(\xi)=\dd_1(\xi)\dd_1(\xi)^*$ is the scalar
\begin{equation}\label{eq:facelieb}
-\Delta_\F(\xi_1,\xi_2)=8-4\cos(2\pi\xi_1)-4\cos(2\pi\xi_2)\ .
\end{equation}
This is twice the standard graph-Laplacian on $\Z^2$, namely the Laplacian of $\Z^2$
with every edge doubled; the doubling is visible in $\dd_1(\xi)$, whose first and third
entries coincide, as do its second and fourth.
In particular, $\sigma(-\Delta_\F)=[0,16]$ consists of a single dispersive band, and
$-\Delta_\F$ has no flat band.
 
We thus conclude from Theorem \ref{theo:Dm} that $D_m$ on the Lieb lattice has exactly
two flat bands, located at $\pm\sqrt{2+m^2}$, both inherited from $-\Delta_\V$.
In contrast with the Kagome and Super-Kagome lattices studied in
Section \ref{subsec:kagomeandsuper}, the faces of the Lieb lattice do not contribute to any new flat band.

\section{Localised eigenvectors}\label{sec_loc}

In this final section we make explicit the connection provided in Theorem \ref{theo:Dm} 
by relating each localised eigenvector corresponding to an eigenvalue $\lambda$ of either $-\Delta_\V$ or $-\Delta_\F$ 
to a pair of localised eigenvectors of $D_m$ corresponding to the eigenvalue $\pm\sqrt{\lambda+m^2}$. 
Recall that for a connected infinite graph, every eigenvalue of $-\Delta_\V$ is strictly positive. 
An analogous result can be obtained for $-\Delta_\F$ by noticing that $-\Delta_\F h\equiv 0$ implies 
$\dd_1^* h \equiv 0$ and hence $h$ must take opposite values in faces sharing an oriented edge. 
Since each planar graph defines a tessellation of the plane, we obtain that $h\equiv 0$.

\begin{Lemma}\label{lem:eigenvectorsDm}
Let $X=(\V,\A,\F)$ be a planar graph with bounded geometry, and let $m\geq0$.
\begin{enumerate}
\item[(i)] Let $\lambda > 0$ and $f\in\ell^2(\V)$ be such that
$
-\Delta_\V f=\dd_0^*\dd_0 f=\lambda f\ .
$
Then the two vectors
$$
\Psi_\pm:=\big(f,\ \alpha_\pm\, \dd_0f,\ 0\big)\in \H_{DH}(X)\ ,
\qquad
\alpha_\pm:=\frac{-m\pm\sqrt{m^2+\lambda}}{\lambda}\ ,$$
satisfy $D_m\Psi_\pm=\pm\sqrt{\lambda+m^2}\,\Psi_\pm$.

\item[(ii)] Let $\nu > 0$ and $h\in\ell^2(\F)$ be such that
$
-\Delta_\F h=\dd_1\dd_1^*h=\nu h\ .
$
Then the two vectors
$$
\Phi_\pm:=\big(0,\ \beta_\pm\, \dd_1^*h,\ h\big)\in \H_{DH}(X)\ ,
\qquad
\beta_\pm:=\frac{-m\pm\sqrt{m^2+\nu}}{\nu}\ ,
$$
satisfy $D_m\Phi_\pm=\pm\sqrt{\nu+m^2}\,\Phi_\pm$.
\end{enumerate}
\end{Lemma}

\begin{proof}
(i)
We look for an eigenvector of the form  $(f,\alpha\dd_0f,0)$, for an unknown $\alpha\in\R$. One computes directly, using \eqref{eq:Dmdef},
$$
D_m\begin{pmatrix} f\\ \alpha\, \dd_0f\\0\end{pmatrix}
=\begin{pmatrix} m & \dd_0^{*} & 0\\
\dd_0 & -m &\dd_1^* \\
0 & \dd_1 & m
\end{pmatrix}
\begin{pmatrix} f\\ \alpha\, \dd_0f\\0\end{pmatrix}
=\begin{pmatrix} (m+\alpha\lambda)f\\ (1-m\alpha)\,\dd_0f\\ \alpha\,\dd_1\dd_0f\end{pmatrix}\ .
$$
The last entry vanishes since $\dd_1\dd_0=0$. The first two blocks impose
$$
m+\alpha\lambda=\frac{1-m\alpha}{\alpha}\ ,
\qquad\text{i.e.}\qquad
\lambda\alpha^2+2m\alpha-1=0\ .
$$
Solving this quadratic equation in $\alpha$ gives $\alpha_\pm=\dfrac{-m\pm\sqrt{m^2+\lambda}}{\lambda}$, and the corresponding eigenvalue is
$$
m+\alpha_\pm\lambda=m+\big(-m\pm\sqrt{m^2+\lambda}\big)=\pm\sqrt{\lambda+m^2}\ ,
$$
as announced.

(ii) The proof is similar, by looking for an eigenvector of the form $(0,\beta\dd_1^*h,h)$.
\end{proof}

We now illustrate this result by exhibiting some localised eigenvectors for Kagome and super-Kagome Lattices.
We start by presenting an eigenstate for the Kagome lattice. It is known that an eigenstate of $-\Delta_\V$ is provided by a function supported on a hexagon, alternating $1$ and $-1$. Applying Lemma \ref{lem:eigenvectorsDm}, we obtain $\Psi_\pm$ given by Figure \ref{fig:PsiKagome}.

\newcommand{\KagomeBG}{
  \begin{scope}[opacity=0.20, scale=.3]
    \foreach \i in {-2,...,6}{
      \foreach \j in {-2,...,6}{
        \pgfmathsetmacro{\ii}{2*\i}
        \pgfmathsetmacro{\jj}{4*\j*0.86602540378}
        \begin{scope}[xshift=\ii cm,yshift=\jj cm]
          \draw(0:1)--(60:1)--(120:1)--(180:1)--(240:1)--(300:1)--(360:1);
          \begin{scope}[xshift=1cm,yshift={2*0.86602540378cm}]
            \draw(0:1)--(60:1)--(120:1)--(180:1)--(240:1)--(300:1)--(360:1);
          \end{scope}
        \end{scope}
      }
    }
  \end{scope}
}

\newcommand{\HexCoords}{
  \coordinate (v1) at (0.15, 0.26);
  \coordinate (v2) at (0.45, 0.26);
  \coordinate (v3) at (0.60, 0.52);
  \coordinate (v4) at (0.45, 0.78);
  \coordinate (v5) at (0.15, 0.78);
  \coordinate (v6) at (0.00, 0.52);
  \coordinate (w1) at (0.30, 0.00);
  \coordinate (w2) at (0.75, 0.26);
  \coordinate (w3) at (0.75, 0.78);
  \coordinate (w4) at (0.30, 1.04);
  \coordinate (w5) at (-0.15, 0.78);
  \coordinate (w6) at (-0.15, 0.26);
}

\tikzset{
  vert/.style={circle, fill=black, inner sep=1.8pt},
  hexedge/.style={line width=1.0pt,
    decoration={markings,mark=at position 0.55 with {\arrow{Stealth[length=2.2mm]}}},
    postaction=decorate},
  spoke/.style={line width=0.8pt,
    decoration={markings,mark=at position 0.6575 with {\arrow{Stealth[length=1.8mm]}}},
    postaction=decorate},
  spokerev/.style={line width=0.8pt,
  decoration={markings,mark=at position 0.4175 with {\arrow{Stealth[length=1.8mm]}}},
  postaction=decorate},  
  lbl/.style={font=\Large},
  slbl/.style={font=\Large}
}

\newcommand{\DrawSpokes}{
  \draw[spoke] (v2)--(w1);
  \draw[spokerev] (w1)--(v1);
  \draw[spoke] (v3)--(w2);
  \draw[spokerev] (w2)--(v2);
  \draw[spoke] (v4)--(w3);
  \draw[spokerev] (w3)--(v3);
  \draw[spoke] (v5)--(w4);
  \draw[spokerev] (w4)--(v4);
  \draw[spoke] (v6)--(w5);
  \draw[spokerev] (w5)--(v5);
  \draw[spoke] (v1)--(w6);
  \draw[spokerev] (w6)--(v6);
}
\newcommand{\SpokeLabelsB}{
  \node[slbl] at (0.4365,0.1465) {$+1$};
  \node[slbl] at (0.1635,0.1465) {$+1$};
  \node[slbl] at (0.6915,0.4515) {$-1$};
  \node[slbl] at (0.5550,0.2150) {$-1$};
  \node[slbl] at (0.5550,0.8250) {$+1$};
  \node[slbl] at (0.6915,0.5885) {$+1$};
  \node[slbl] at (0.1635,0.8935) {$-1$};
  \node[slbl] at (0.4365,0.8935) {$-1$};
  \node[slbl] at (-0.0915,0.5885) {$+1$};
  \node[slbl] at (0.0450,0.8250) {$+1$};
  \node[slbl] at (0.0450,0.2150) {$-1$};
  \node[slbl] at (-0.0915,0.4515) {$-1$};
}
\newcommand{\SpokeLabelsC}{
  \node[slbl] at (0.4365,0.1465) {$+\alpha_\pm$};
  \node[slbl] at (0.1635,0.1465) {$+\alpha_\pm$};
  \node[slbl] at (0.6915,0.4515) {$-\alpha_\pm$};
  \node[slbl] at (0.5550,0.2150) {$-\alpha_\pm$};
  \node[slbl] at (0.5550,0.8250) {$+\alpha_\pm$};
  \node[slbl] at (0.6915,0.5885) {$+\alpha_\pm$};
  \node[slbl] at (0.1435,0.8935) {$-\alpha_\pm$};
  \node[slbl] at (0.4365,0.8935) {$-\alpha_\pm$};
  \node[slbl] at (-0.0915,0.5885) {$+\alpha_\pm$};
  \node[slbl] at (0.0450,0.8250) {$+\alpha_\pm$};
  \node[slbl] at (0.0450,0.2150) {$-\alpha_\pm$};
  \node[slbl] at (-0.0915,0.4515) {$-\alpha_\pm$};
}

\begin{figure}[!htb]
\centering
\begin{subfigure}{0.31\textwidth}
\centering
\resizebox{\linewidth}{!}{\begin{tikzpicture}[scale=8]
  \useasboundingbox (-0.24,-0.05) rectangle (0.84,1.09);
  \clip (-0.24,-0.05) rectangle (0.84,1.09);
  \KagomeBG
  \HexCoords
  \draw[line width=1.0pt] (v1)--(v2)--(v3)--(v4)--(v5)--(v6)--cycle;
  \node[vert,label={[font=\Large]above right:$+1$}] at (v1) {};
  \node[vert,label={[font=\Large]above left:$-1$}] at (v2) {};
  \node[vert,label={[font=\Large]left:$+1$}] at (v3) {};
  \node[vert,label={[font=\Large]below left:$-1$}] at (v4) {};
  \node[vert,label={[font=\Large]below right:$+1$}] at (v5) {};
  \node[vert,label={[font=\Large]right:$-1$}] at (v6) {};
\end{tikzpicture}}
\caption{$f$}
\end{subfigure}
\hfill
\begin{subfigure}{0.31\textwidth}
\centering
\resizebox{\linewidth}{!}{\begin{tikzpicture}[scale=8]
  \useasboundingbox (-0.24,-0.05) rectangle (0.84,1.09);
  \clip (-0.24,-0.05) rectangle (0.84,1.09);
  \KagomeBG
  \HexCoords
  \draw[hexedge] (v1)--(v2);
  \draw[hexedge] (v2)--(v3);
  \draw[hexedge] (v3)--(v4);
  \draw[hexedge] (v4)--(v5);
  \draw[hexedge] (v5)--(v6);
  \draw[hexedge] (v6)--(v1);
  \DrawSpokes
  \SpokeLabelsB
  \foreach \i in {1,...,6}{ \node[vert,fill=gray!45,inner sep=1.4pt] at (v\i) {}; }
  \foreach \i in {1,...,6}{ \node[circle,fill=gray!45,inner sep=1.1pt] at (w\i) {}; }
  \node[lbl] at (0.3000,0.2100) {$-2$};
  \node[lbl] at (0.5683,0.3650) {$+2$};
  \node[lbl] at (0.5683,0.6750) {$-2$};
  \node[lbl] at (0.3000,0.8300) {$+2$};
  \node[lbl] at (0.0317,0.6750) {$-2$};
  \node[lbl] at (0.0317,0.3650) {$+2$};
\end{tikzpicture}}
\caption{$\dd_0 f$}
\end{subfigure}
\hfill
\begin{subfigure}{0.31\textwidth}
\centering
\resizebox{\linewidth}{!}{\begin{tikzpicture}[scale=8]
  \useasboundingbox (-0.24,-0.05) rectangle (0.84,1.09);
  \clip (-0.24,-0.05) rectangle (0.84,1.09);
  \KagomeBG
  \HexCoords
  \draw[hexedge] (v1)--(v2);
  \draw[hexedge] (v2)--(v3);
  \draw[hexedge] (v3)--(v4);
  \draw[hexedge] (v4)--(v5);
  \draw[hexedge] (v5)--(v6);
  \draw[hexedge] (v6)--(v1);
  \DrawSpokes
  \SpokeLabelsC
  \node[vert,label={[font=\Large]above right:$+1$}] at (v1) {};
  \node[vert,label={[font=\Large]above left:$-1$}] at (v2) {};
  \node[vert,label={[font=\Large]left:$+1$}] at (v3) {};
  \node[vert,label={[font=\Large]below left:$-1$}] at (v4) {};
  \node[vert,label={[font=\Large]below right:$+1$}] at (v5) {};
  \node[vert,label={[font=\Large]right:$-1$}] at (v6) {};
  \node[lbl] at (0.3000,0.2100) {$-2\alpha_\pm$};
  \node[lbl] at (0.5793,0.3600) {$2\alpha_\pm$};
  \node[lbl] at (0.5793,0.6850) {$-2\alpha_\pm$};
  \node[lbl] at (0.3000,0.8300) {$2\alpha_\pm$};
  \node[lbl] at (0.0037,0.6850) {$-2\alpha_\pm$};
  \node[lbl] at (0.0187,0.3600) {$2\alpha_\pm$};
\end{tikzpicture}}
\caption{$\Psi_\pm$}\label{fig:PsiKagome}
\end{subfigure}
\caption{The flat-band states for the Kagome lattice: (a) the eigenvector $f$ of $-\Delta_\V$ with eigenvalue $6$; (b) $\dd_0 f$, with the edge orientation indicated by the arrows; (c) the resulting eigenvectors $\Psi_\pm=(f,\alpha_\pm\dd_0f,0)$ of $D_m$, with $\alpha_\pm=\dfrac{-m\pm\sqrt{m^2+6}}{6}$.}
\label{fig:kagome_flatband}
\end{figure}
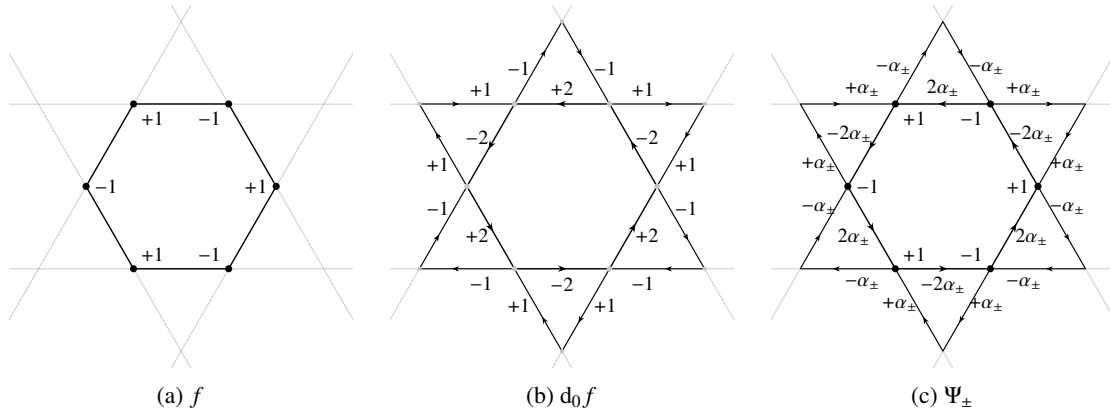

To give an example for the Super-kagome lattice, we start with an eigenstate for $-\Delta_\F$. This can be easily obtained by alterning orientations of triangles around a single dodecagon. The construction is presented in Figure \ref{fig:superkagome_flatband},
with explanations provided in the caption.

\newcommand{\DodecaCoords}{
  \coordinate (P0) at (1.8660,0.5000);
  \coordinate (P1) at (1.3660,1.3660);
  \coordinate (P2) at (0.5000,1.8660);
  \coordinate (P3) at (-0.5000,1.8660);
  \coordinate (P4) at (-1.3660,1.3660);
  \coordinate (P5) at (-1.8660,0.5000);
  \coordinate (P6) at (-1.8660,-0.5000);
  \coordinate (P7) at (-1.3660,-1.3660);
  \coordinate (P8) at (-0.5000,-1.8660);
  \coordinate (P9) at (0.5000,-1.8660);
  \coordinate (P10) at (1.3660,-1.3660);
  \coordinate (P11) at (1.8660,-0.5000);
  \coordinate (Q0) at (2.3660,1.3660);
  \coordinate (Q1) at (-0.0000,2.7321);
  \coordinate (Q2) at (-2.3660,1.3660);
  \coordinate (Q3) at (-2.3660,-1.3660);
  \coordinate (Q4) at (-0.0000,-2.7321);
  \coordinate (Q5) at (2.3660,-1.3660);
}

\tikzset{
  dodec/.style={line width=1.0pt},
  tri/.style={line width=1.0pt},
  trifill/.style={fill=white},
  hexedge/.style={line width=0.9pt,
    decoration={markings,mark=at position 0.55 with {\arrow{Stealth[length=2mm]}}},
    postaction=decorate},
  lbl/.style={font=\small},
  slbl/.style={font=\small},
  rot/.style={font=\normalsize}
}

\newcommand{\DrawSkeleton}{
  \draw[dodec, fill=black!8] (P0)--(P1)--(P2)--(P3)--(P4)--(P5)--(P6)--(P7)--(P8)--(P9)--(P10)--(P11)--cycle;
  \draw[tri, trifill] (P0)--(P1)--(Q0)--cycle;
  \draw[tri, trifill] (P2)--(P3)--(Q1)--cycle;
  \draw[tri, trifill] (P4)--(P5)--(Q2)--cycle;
  \draw[tri, trifill] (P6)--(P7)--(Q3)--cycle;
  \draw[tri, trifill] (P8)--(P9)--(Q4)--cycle;
  \draw[tri, trifill] (P10)--(P11)--(Q5)--cycle;
  \draw[gray!45] (Q0) -- (2.6691,1.5410);
  \draw[gray!45] (Q1) -- (-0.0000,3.0821);
  \draw[gray!45] (Q2) -- (-2.6691,1.5410);
  \draw[gray!45] (Q3) -- (-2.6691,-1.5410);
  \draw[gray!45] (Q4) -- (-0.0000,-3.0821);
  \draw[gray!45] (Q5) -- (2.6691,-1.5410);
}

\newcommand{\DrawArrowsB}{
  \draw[hexedge] (P0)--(Q0); \draw[hexedge] (Q0)--(P1); \draw[hexedge] (P1)--(P0);
  \draw[hexedge] (P2)--(P3); \draw[hexedge] (P3)--(Q1); \draw[hexedge] (Q1)--(P2);
  \draw[hexedge] (P4)--(Q2); \draw[hexedge] (Q2)--(P5); \draw[hexedge] (P5)--(P4);
  \draw[hexedge] (P6)--(P7); \draw[hexedge] (P7)--(Q3); \draw[hexedge] (Q3)--(P6);
  \draw[hexedge] (P8)--(Q4); \draw[hexedge] (Q4)--(P9); \draw[hexedge] (P9)--(P8);
  \draw[hexedge] (P10)--(P11); \draw[hexedge] (P11)--(Q5); \draw[hexedge] (Q5)--(P10);
}

\newcommand{\CircArrows}{
  \node[rot] at (1.8660,1.0774) {$\circlearrowleft$};
  \node[rot] at (-0.0000,2.1547) {$\circlearrowright$};
  \node[rot] at (-1.8660,1.0774) {$\circlearrowleft$};
  \node[rot] at (-1.8660,-1.0774) {$\circlearrowright$};
  \node[rot] at (-0.0000,-2.1547) {$\circlearrowleft$};
  \node[rot] at (1.8660,-1.0774) {$\circlearrowright$};
}

\newcommand{\ZeroLabels}{
  \node[lbl] at (0,0) {$0$};
  \node[lbl] at (2.9722,1.7160) {$0$};
  \node[lbl] at (-0.0000,3.4321) {$0$};
  \node[lbl] at (-2.9722,1.7160) {$0$};
  \node[lbl] at (-2.9722,-1.7160) {$0$};
  \node[lbl] at (-0.0000,-3.4321) {$0$};
  \node[lbl] at (2.9722,-1.7160) {$0$};
}

\begin{figure}[!ht]
\centering
\begin{subfigure}{0.31\textwidth}
\centering
\resizebox{\linewidth}{!}{\begin{tikzpicture}[scale=1]
  \DodecaCoords
  \DrawSkeleton
  \CircArrows
\end{tikzpicture}}
\caption{$h$}
\end{subfigure}
\hfill
\begin{subfigure}{0.31\textwidth}
\centering
\resizebox{\linewidth}{!}{\begin{tikzpicture}[scale=1]
  \DodecaCoords
  \draw[dodec, fill=black!8] (P0)--(P1)--(P2)--(P3)--(P4)--(P5)--(P6)--(P7)--(P8)--(P9)--(P10)--(P11)--cycle;
  \draw[gray!45] (Q0) -- (2.6691,1.5410);
  \draw[gray!45] (Q1) -- (-0.0000,3.0821);
  \draw[gray!45] (Q2) -- (-2.6691,1.5410);
  \draw[gray!45] (Q3) -- (-2.6691,-1.5410);
  \draw[gray!45] (Q4) -- (-0.0000,-3.0821);
  \draw[gray!45] (Q5) -- (2.6691,-1.5410);
  \DrawArrowsB
\end{tikzpicture}}
\caption{$\dd_1^*h$}
\end{subfigure}
\hfill
\begin{subfigure}{0.31\textwidth}
\centering
\resizebox{\linewidth}{!}{\begin{tikzpicture}[scale=1]
  \DodecaCoords
  \draw[dodec, fill=black!8] (P0)--(P1)--(P2)--(P3)--(P4)--(P5)--(P6)--(P7)--(P8)--(P9)--(P10)--(P11)--cycle;
  \draw[gray!45] (Q0) -- (2.6691,1.5410);
  \draw[gray!45] (Q1) -- (-0.0000,3.0821);
  \draw[gray!45] (Q2) -- (-2.6691,1.5410);
  \draw[gray!45] (Q3) -- (-2.6691,-1.5410);
  \draw[gray!45] (Q4) -- (-0.0000,-3.0821);
  \draw[gray!45] (Q5) -- (2.6691,-1.5410);
  \DrawArrowsB
  \CircArrows
  \node[slbl] at (1.3389,0.7730) {$\beta_\pm$};
  \node[slbl] at (1.8660,1.6860) {$\beta_\pm$};
  \node[slbl] at (2.3932,0.7730) {$\beta_\pm$};
  \node[slbl] at (-0.0000,1.5460) {$\beta_\pm$};
  \node[slbl] at (-0.5271,2.4590) {$\beta_\pm$};
  \node[slbl] at (0.5271,2.4590) {$\beta_\pm$};
  \node[slbl] at (-1.3389,0.7730) {$\beta_\pm$};
  \node[slbl] at (-2.3932,0.7730) {$\beta_\pm$};
  \node[slbl] at (-1.8660,1.6860) {$\beta_\pm$};
  \node[slbl] at (-1.3389,-0.7730) {$\beta_\pm$};
  \node[slbl] at (-1.8660,-1.6860) {$\beta_\pm$};
  \node[slbl] at (-2.3932,-0.7730) {$\beta_\pm$};
  \node[slbl] at (-0.0000,-1.5460) {$\beta_\pm$};
  \node[slbl] at (0.5271,-2.4590) {$\beta_\pm$};
  \node[slbl] at (-0.5271,-2.4590) {$\beta_\pm$};
  \node[slbl] at (1.3389,-0.7730) {$\beta_\pm$};
  \node[slbl] at (2.3932,-0.7730) {$\beta_\pm$};
  \node[slbl] at (1.8660,-1.6860) {$\beta_\pm$};
\end{tikzpicture}}
\caption{$\Phi_\pm$}
\end{subfigure}
\caption{The flat-band states for the Super-Kagome lattice: (a) the eigenvector $h$ of $-\Delta_\F$ with
eigenvalue $3$, the orientation of each triangular face is represented and the value $1$ (not represented) is attached to each of them;
(b) $\dd_1^*h$: each triangle's three edges are oriented according to its own
orientation from (a), and all carry the same value $1$ (not represented); (c) the resulting eigenvectors
$\Phi_\pm=(0,\beta_\pm\dd_1^*h,h)$ of $D_m$,  with the value $1$ attached to each triangular face (not represented), and with
$\beta_\pm=\dfrac{-m\pm\sqrt{m^2+3}}{3}$ .}
\label{fig:superkagome_flatband}
\end{figure}
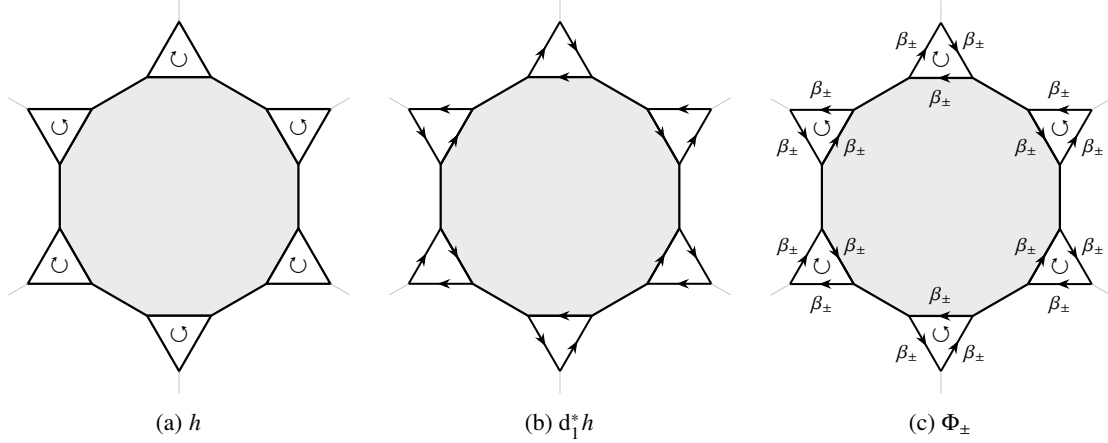

\appendix
\section{Duality approach}

In this appendix, we mention an alternative approach for studying the spectrum of Dirac--Hodge
operators on planar graphs based on duality. Indeed, studying the operator $-\Delta_\F$ on a planar graph
is equivalent to studying the more common operator $-\Delta_{\V}$ on its dual graph. 
Let us fix the setting, by using the notations introduced in Section \ref{subsec_planar}.

\begin{Definition}[Dual graph]\label{def:dual}
Let $(\V,\E,\CC)$ be the unoriented data for the planar graph $X=(\V,\A,\F)$.
The \emph{dual graph} of $X$ is the planar graph
$X^*=(\V^*,\A^*,\F^*)$ defined by the following unoriented data
$(\V^*,\E^*,\CC^*)$~:
\begin{itemize}
\item The vertex set of $X^*$ is the set of unoriented faces of $X$, \ie $\V^*:=\CC$.
\item For any $\e\in \E$ such that $\e\in\f_1$ and $\e\in\f_2$, with $\f_1,\f_2 \in \CC$,
an unoriented edge $\e^*\in \E^*$ joining the vertices $\f_1^*$ and $\f_2^*$ is created. 
\item For each vertex $\x\in\V$, the cyclic order of edges belonging to $\E_\x$
defines a chordless cycle in $X^*$, and consequently a unique element of $\CC^*$.
\end{itemize}
\end{Definition}

Note that the third item leads to a one-to-one correspondence between $\V$ and $\CC^*$.
As expected, one gets that $X^{**}=X$. An example of a graph
and its dual graph is presented in Figure \ref{fig:kagomedual}.
Then, by comparing the expression in \eqref{eq:Deltavertices} with the expression in \eqref{eq:Deltafaces} 
we obtain the following statement.

\begin{figure}[!ht]
\centering
\resizebox{\textwidth}{!}{%
\begin{tikzpicture}[scale=1.6]

\def\sq{1.73205}   
\def\sqh{0.86603}  
\def\sqsi{0.28868} 
\def\sqfi{1.44338} 

\def\xmin{0.3}
\def\xmax{6.7}
\def\ymin{1.2}
\def\ymax{6.3}

\begin{scope}
  \clip (\xmin,\ymin) rectangle (\xmax,\ymax);
  \foreach \ii in {-2,...,4}{
    \foreach \jj in {-2,...,4}{
      \pgfmathsetmacro{\Ax}{\ii*2+\jj}
      \pgfmathsetmacro{\Ay}{\jj*\sq}
      \pgfmathsetmacro{\Bx}{\ii*2+\jj+1}
      \pgfmathsetmacro{\Cx}{\ii*2+\jj+0.5}
      \pgfmathsetmacro{\Cy}{\jj*\sq+\sqh}
      \pgfmathsetmacro{\Axr}{\ii*2+\jj+2}
      \pgfmathsetmacro{\Axu}{\ii*2+\jj+1}
      \pgfmathsetmacro{\Ayu}{(\jj+1)*\sq}
      \pgfmathsetmacro{\Bxu}{\ii*2+\jj}
      \draw[thick] (\Ax,\Ay)--(\Bx,\Ay);
      \draw[thick] (\Ax,\Ay)--(\Cx,\Cy);
      \draw[thick] (\Bx,\Ay)--(\Cx,\Cy);
      \draw[thick] (\Bx,\Ay)--(\Axr,\Ay);
      \draw[thick] (\Cx,\Cy)--(\Axu,\Ayu);
      \draw[thick] (\Cx,\Cy)--(\Bxu,\Ayu);
    }
  }
  \foreach \ii in {-1,...,4}{
    \foreach \jj in {-1,...,4}{
      \pgfmathsetmacro{\Ax}{\ii*2+\jj}
      \pgfmathsetmacro{\Ay}{\jj*\sq}
      \pgfmathsetmacro{\Bx}{\ii*2+\jj+1}
      \pgfmathsetmacro{\Cx}{\ii*2+\jj+0.5}
      \pgfmathsetmacro{\Cy}{\jj*\sq+\sqh}
      \filldraw[black] (\Ax,\Ay) circle (1.8pt);
      \filldraw[black] (\Bx,\Ay) circle (1.8pt);
      \filldraw[black] (\Cx,\Cy) circle (1.8pt);
    }
  }
\end{scope}

\begin{scope}[xshift=8.0cm]
  \clip (\xmin,\ymin) rectangle (\xmax,\ymax);
  \foreach \ii in {-2,...,4}{
    \foreach \jj in {-2,...,4}{
      \pgfmathsetmacro{\Tux}{\ii*2+\jj+0.5}
      \pgfmathsetmacro{\Tuy}{\jj*\sq+\sqsi}
      \pgfmathsetmacro{\Tdx}{\ii*2+\jj+0.5}
      \pgfmathsetmacro{\Tdy}{\jj*\sq+\sqfi}
      \pgfmathsetmacro{\Hcx}{\ii*2+\jj+1.5}
      \pgfmathsetmacro{\Hcy}{\jj*\sq+\sqh}
      \pgfmathsetmacro{\Hbx}{\ii*2+\jj+0.5}
      \pgfmathsetmacro{\Hby}{(\jj-1)*\sq+\sqh}
      \pgfmathsetmacro{\Hlx}{\ii*2+\jj-0.5}
      \pgfmathsetmacro{\Hly}{\jj*\sq+\sqh}
      \pgfmathsetmacro{\Hux}{\ii*2+\jj+0.5}
      \pgfmathsetmacro{\Huy}{(\jj+1)*\sq+\sqh}
      \draw[thick] (\Tux,\Tuy)--(\Hbx,\Hby);
      \draw[thick] (\Tux,\Tuy)--(\Hcx,\Hcy);
      \draw[thick] (\Tux,\Tuy)--(\Hlx,\Hly);
      \draw[thick] (\Tdx,\Tdy)--(\Hcx,\Hcy);
      \draw[thick] (\Tdx,\Tdy)--(\Hux,\Huy);
      \draw[thick] (\Tdx,\Tdy)--(\Hlx,\Hly);
    }
  }
  \foreach \ii in {-2,...,4}{
    \foreach \jj in {-2,...,4}{
      \pgfmathsetmacro{\Tux}{\ii*2+\jj+0.5}
      \pgfmathsetmacro{\Tuy}{\jj*\sq+\sqsi}
      \pgfmathsetmacro{\Tdx}{\ii*2+\jj+0.5}
      \pgfmathsetmacro{\Tdy}{\jj*\sq+\sqfi}
      \pgfmathsetmacro{\Hx}{\ii*2+\jj+1.5}
      \pgfmathsetmacro{\Hy}{\jj*\sq+\sqh}
      \filldraw[black] (\Tux,\Tuy) circle (1.8pt);
      \filldraw[black] (\Tdx,\Tdy) circle (1.8pt);
      \filldraw[black] (\Hx,\Hy)   circle (1.8pt);
    }
  }
\end{scope}
\end{tikzpicture}}
\caption{The Kagome lattice $(3.6)^2$ and its dual, the Rhombille lattice. The Rhombille is a bipartite graph consisting of vertices corresponding to hexagons, which are connected to six vertices corresponding to triangles. The ones corresponding to triangles are each connected to three vertices corresponding to hexagons.}
\label{fig:kagomedual}
\end{figure}
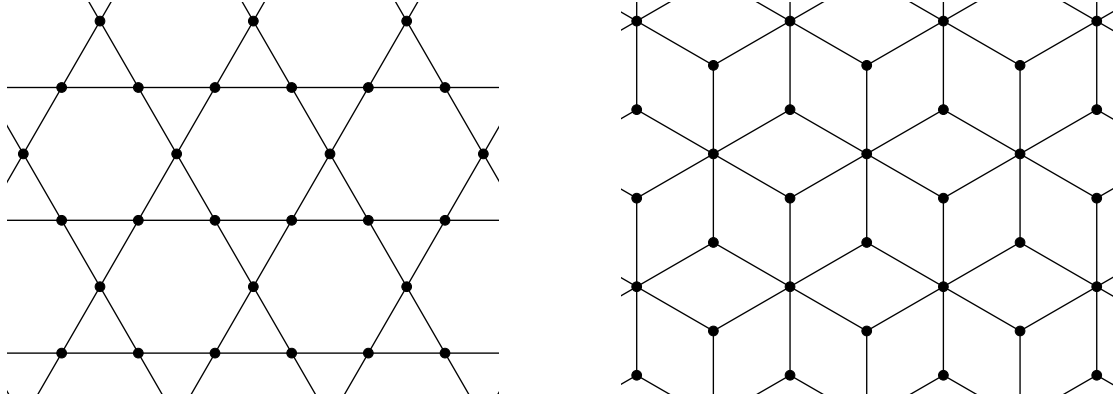

\begin{Proposition}\label{prop_dual}
For any planar graph $X$, its Laplacian on faces $-\Delta_\F$ is unitarily equivalent to the graph-Laplacian $-\Delta_{\V^*}$ 
of its dual graph.
\end{Proposition}

Let us finally mention that this statement leads directly to the result about the spectra for the Dirac--Hodge
operators on a subfamily of Archimedean lattices. 
More precisely, for the three Platonic lattices (the square lattice $4^4$, the triangular lattices $3^6$, and the hexagonal lattice $6^3$)
the investigations can be shortened by taking the duality argument of Proposition \ref{prop_dual} into account.
Indeed, observe firstly that $(4^4)^*=(4^4)$ and that $(3^6)^*=(6^3)$.
Thus, $-\Delta_\F\big(4^4\big)$ is unitarily equivalent to $-\Delta_\V\big(4^4\big)$, 
$-\Delta_\F\big(3^6\big)$ is unitarily equivalent to $-\Delta_\V\big(6^3\big)$, 
and reciprocally $-\Delta_\F\big(6^3\big)$ is unitarily equivalent to $-\Delta_\V\big(3^6\big)$.
As a consequence, the absence of flat bands for $-\Delta_\V$ already obtained in \cite{KeTaWi24}
lead to the absence of flat bands for the operators $-\Delta_\F$, 
and a fortiori for the corresponding operators $D_m$.

\end{document}